\documentclass[11pt]{amsart}
\usepackage[utf8]{inputenc}
\usepackage{amsmath,amssymb,amsxtra,anysize,xcolor,tikz,tikz-cd,combelow,hyperref}
\usetikzlibrary{arrows,decorations.markings}

\newcommand{\GCD}{\mathrm{GCD}}
\newcommand{\C}{\mathbb{C}}
\newcommand{\Z}{\mathbb{Z}}

\newcommand{\cO}{\mathcal{O}}
\newcommand{\cE}{\mathcal{E}}

\newcommand{\dinv}{\mathrm{dinv}}
\newcommand{\codinv}{\mathrm{codinv}}
\newcommand{\area}{\mathrm{area}}
\newcommand{\SYT}{\mathrm{SYT}}

\newcommand{\Gaps}{\mathrm{Gaps}}
\newcommand{\Dyck}{\mathrm{Dyck}}
\newcommand{\Gen}{\mathrm{Gen}}
\newcommand{\Inv}{\mathrm{Inv}}
\newcommand{\Ord}{\mathrm{Ord}}

\newtheorem{theorem}{Theorem}[section]
\newtheorem{corollary}[theorem]{Corollary} 
\newtheorem{lemma}[theorem]{Lemma} 
\newtheorem{proposition}[theorem]{Proposition} 
\newtheorem{example}[theorem]{Example}
\newtheorem{conjecture}[theorem]{Conjecture}
  
\newtheorem{remark}[theorem]{Remark}  
\newtheorem{definition}[theorem]{Definition}

\newcommand{\free}{\mathrm{free}}

\author{Eugene Gorsky}
\address{Department of Mathematics, University of California Davis\newline One Shields Avenue, Davis CA 95616 USA}
\email{egorskiy@math.ucdavis.edu}
\author{Mikhail Mazin}
\address{Mathematics Department, Kansas State University\newline Cardwell Hall, 1228 N M.L.K. Jr. Dr Suite 138, Manhattan, KS 66506 USA}
\email{mmazin@math.ksu.edu}
\author{Alexei Oblomkov}
\address{Department of Mathematics and Statistics,
University of Massachusetts\newline
Lederle Graduate Research Tower, 710 North Pleasant Street Amherst, MA 01003-9305 USA}
\email{oblomkov@math.umass.edu}

\title{Generic curves and non-coprime Catalans}

\begin{document}

\keywords{Compactified Jacobians, Dyck paths, $q,t$-Catalan numbers}
\subjclass{14H20,14H40,05E14}

\begin{abstract}
We compute the Poincar\'e polynomials of the compactified Jacobians for plane curve singularities with Puiseux exponents $(nd,md,md+1)$, and relate them to the combinatorics of $q,t$-Catalan numbers in the non-coprime case. We also confirm a conjecture of Cherednik and Danilenko for such curves.

\end{abstract}

\maketitle

\section{Introduction}

In this paper, we study the topology of compactified Jacobians of plane curve singularities. We focus on the case where the curve is reduced and locally irreducible (or {\em unibranched}), and it is known \cite{AIK,AK} that in this case compactified Jacobian is irreducible as well. Throughout the paper we work over $\mathbb{C}$. We explain where  this assumption is used, in case our reader is curious about the extension of the result to other fileds.

Compactified Jacobians play an important role in modern geometric representation theory. First, they are closely related to Hilbert schemes of points on singular curves, singular fibers in the Hitchin fibration and  affine Springer fibers. In particular, counting points in the compactified Jacobians over finite fields is related to certain orbital integrals \cite{Laumon,KivTsai,Tsai}. Recent works \cite{GK,HKW,GKO} relate them to the representation theory of Coulomb branch algebras, defined by Braverman, Finkelberg and Nakajima \cite{BFN}.

Second, a set of conjectures of the third author, Rasmussen and Shende \cite{OS,ORS} relates the homology of compactified Jacobians to the {\em Khovanov-Rozansky homology} of the corresponding knots and links. In particular, the conjectures imply that the homology of the compactified Jacobian is expected to be determined by the topology of the link or, in the unibranched case, by the collection of Puiseux pairs of the singularity.

The progress in explicit computations of the homology of compactified Jacobians has been quite slow. For the quasi-homogeneous curves $C=\{x^m=y^n\}, \GCD(m,n)=1$, the homology was computed in many sources, starting with Lusztig and Smelt \cite{LS}. The key observation is that in this case $\overline{JC}$ admits a paving by affine cells. These cells and their dimensions have been given a number of combinatorial interpretations in \cite{GM1,GM2,GMV,Hikita}, where they were related to $q,t$-Catalan combinatorics. In \cite{OY1,OY2} the third author and Yun determined the ring structure on the homology in this case.

In a different direction, Piontkowski \cite{Piont} have computed the homology of compactified Jacobians for plane curve singularities with one Puiseux pair defined by the parametrization $(x,y)=(t^n,t^m+\ldots), \GCD(m,n)=1$. He showed that $\overline{JC}$ again admits an affine paving, and the combinatorics of cells depends only on $(m,n)$ and hence agrees with the quasi-homogeneous case $(x,y)=(t^n,t^m)$. Moreover, Piontkowski computed the cell decompositions for some singularities with two Puiseux exponents (see Section \ref{sec: piont}), where the combinatorics becomes rather subtle. 

In our main theorem, we vastly generalize the results of Piontkowski and prove the following.

\begin{theorem}
\label{thm: intro main}
Suppose $\GCD(m,n)=1$ and $d\ge 1$, consider the plane curve singularity $C$ defined by the parametrization
\begin{equation}
\label{eq: def generic}
(x(z),y(z))=(t^{nd},t^{md}+\lambda t^{{md}+1}+\ldots),\quad \lambda\neq 0
\end{equation}
Then:

a) The compactified Jacobian $\overline{JC}$ admits an affine paving where the cells are in bijection with Dyck paths $D$ in an $(nd)\times (md)$ rectangle. 

b) The Poincar\'e polynomial of $\overline{JC}$ is given by 
$$
P_{\overline{JC}}(T)=\sum_{D\in \Dyck(nd,md)}T^{2(\delta-\dinv(D))}
$$
where $\dinv$ is a certain statistics on Dyck paths defined in Section \ref{sec: combi}, and $\delta$ is the number of boxes weakly under the diagonal in an $(nd)\times (md)$ rectangle.

c) In particular, the Poincar\'e polynomial does not depend on $\lambda$ (as long as it is nonzero) or on the higher order terms in \eqref{eq: def generic}.
\end{theorem}

We call the plane curve singularities \eqref{eq: def generic} {\em generic} (or generic curves in short), since a generic plane curve singularity  with the first Puiseux pair $(nd,md)$ has this form.  In the notations of \cite{Piont}, for $d>1$ these are the plane curve singularities with Puiseux exponents $(nd,md,md+1)$ while for $d=1$ these have Puiseux exponents $(n,m)$. We recall some basic facts about such singularities in Section \ref{sec: background}.

The affine paving in the statement of Theorem \ref{thm: intro main} is obtained by intersecting the compactified Jacobian with the Schubert cell paving of the affine Grassmannian. Thus one can define a partial order on the strata such that  the boundary  of an affine cell  lies in the union of the cells with smaller indices in this order. Hence we conclude:

\begin{corollary}
For generic plane curve singularities, the cohomology  $H^*(\overline{JC})$ is supported in even degrees and the weight filtration on  $H^*(\overline{JC})$ is pure in the sense of \cite{GKM,GKM2}.
\end{corollary}

Very recently, Kivinen and Tsai \cite{KivTsai} used completely different methods ($p$-adic harmonic analysis) to count points in arbitrary compactified Jacobians  over finite fields $\mathbb{F}_q$. They proved that the result is always a polynomial in $q$ and hence recovers the weight polynomial of  $\overline{JC}$ (see Theorem \ref{thm: KivTsai} for more details). Given the above purity result, for generic plane curve singularities the Poincar\'e polynomial of $\overline{JC}$ agrees with the weight polynomial and our result agrees with theirs. 


As a corollary of Theorem \ref{thm: intro main}, we get that the Euler characteristic of $\overline{JC}$ is given by the number of Dyck paths in $(nd)\times (md)$ rectangle. For example, for the singularity $C=(t^4,t^6+t^7)$ the Euler characteristic $\chi(\overline{JC})=23$ is equal to the number of Dyck paths in a $4\times 6$ rectangle (see Example \ref{ex: 467}), in agreement with \cite{Piont}. In general, the combinatorial results of \cite{Bergeron,Bizley} immediately imply the following.

\begin{corollary}
\label{cor: bizley}
Let us fix coprime $m,n$, and for each $d\ge 1$ let $C_d$ be a generic plane curve singularity  \eqref{eq: def generic}. Then
$$
1+\sum_{d=1}^{\infty}z^d\chi\left(\overline{JC_{d}}\right)=\exp\left(\sum_{d=1}^{\infty}\frac{z^d}{(m+n)d}\binom{md+nd}{md}\right).
$$
\end{corollary}

Next, we address the conjectures of Cherednik, Danilenko and Philipp \cite{CheD,CheP}, which proposed an expression for the Poincar\'e polynomials of compactified Jacobians in terms of certain matrix elements of certain operators in the double affine Hecke algebra, see Section \ref{sec: daha} for more details. We are able to prove this conjecture for generic plane curve singularities:

\begin{theorem}
\label{thm: intro daha}
Consider the two-variable polynomial
\begin{equation}
\label{eq: def qtcat}
C_{nd,md}(Q,T)=\sum_{D\in \Dyck(dn,dm)}Q^{\area(D)}T^{\dinv(D)}
\end{equation}
where $\area(D)$ is the number of full boxes between a Dyck path $D$ and the diagonal. It satisfies the following properties:

a) It is symmetric in $Q$ and $T$ : $C_{nd,md}(Q,T)=C_{nd,md}(T,Q)$

b) At $Q=1$ it specializes to the Poincar\'e polynomial of $\overline{JC}$ (up to a linear change of the variable). 

c) It is given by the matrix element $(\gamma_{n,m}(e_d)(1),e_{nd})$ of the elliptic Hall algebra operator $\gamma_{n,m}(e_d)$.

d) It agrees with the Cherednik-Danilenko conjecture (Conjecture \ref{conj: CD}).
\end{theorem} 

Recently, the second author, Caprau, Gonzalez, and Hogancamp were able to show that the polynomial $C_{nd,md}(Q,T)$ is the $a=0$ specialization of the Poincar\'e polynomial of the reduced Khovanov-Rozansky homology of the $(d,mnd+1)$-cable of the $(n,m)$ torus knot (\cite{CGHM}). Combined with Theorem \ref{thm: intro daha}, this confirms a conjecture of the third author, Rasmussen, and Shende in the case of the generic plane curve singularities.

In  Section \ref{sec: daha} we also give a different, manifestly symmetric in $Q,T$ explicit formula for $C_{nd,md}(Q,T)$. The fact that the two formulas agree is a special case of the Compositional Rational Shuffle Theorem, conjectured in \cite{BGLX} and proved in \cite{Mellit}. In fact, parts (a) and (c) of Theorem \ref{thm: intro daha} immediately follow from that theorem, and part (d) follows from (c), see Section \ref{sec: daha} for a detailed proof. Part (b) is  a rephrasing of Theorem \ref{thm: intro main}.

Theorem \ref{thm: intro daha}(a) allows us to give a simple formula for the Poincar\'e polynomial of the compactified Jacobian:

\begin{corollary}
\label{cor: Poincare area intro}
The Poincar\'e polynomial of $\overline{JC}$ equals
$$
T^{2\delta}C_{nd,md}(1,T^{-2})=T^{2\delta}C_{nd,md}(T^{-2},1)=\sum_{D\in \Dyck(nd,md)}T^{2(\delta-\area(D))}.
$$
\end{corollary}

\subsection{Proof strategy and organization of the paper}

Let us comment on the main ideas in the proof  of Theorem \ref{thm: intro main}. The ring of functions $\cO_C=\C[[t^{nd},t^{md}+\lambda t^{md+1}+\ldots]]$ on the plane curve singularity $C$ has a natural valuation given by the order of a function in the parameter $t$. We consider the semigroup $\Gamma$ of $C$ consisting of all possible orders of functions, it is generated by the numbers $(nd,md,mnd+1)$.

By definition, the compactified Jacobian $\overline{JC}$ is the moduli space of torsion free rank one $\cO_C$-modules with appropriate framing, or, equivalently,  $\cO_C$-submodules of $\C((t))$ with appropriate normalization, see Section \ref{sec: background Jacobian}. Clearly, the set of valuations for any such submodule $M$ is a $\Gamma$-submodule $\Delta_M$ of $\Z$.

We stratify the compactified Jacobian by the possible combinatorial types of submodules $\Delta=\Delta_M$:
$$
J_{\Delta}=\{M\subset \C((t)): \cO_CM\subset M,\ \Delta_M=\Delta\}.
$$ 
Next, we introduce the key definition of an {\bf admissible} $\Gamma$-submodule (Definition \ref{def: admissible}) generalizing the constructions of Piontkowski \cite{Piont} and Cherednik-Philipp \cite{CheP}. In Lemma \ref{lem: admissible necessary} we prove that if $\Delta$ is not admissible then $J_{\Delta}$ is empty. The converse statement is significantly harder:

\begin{theorem}
\label{thm: admissible intro}[(Proposition~\ref{prop:adm-tria})]
Suppose that $\Delta$ is admissible. Then the corresponding stratum $J_{\Delta}\subset \overline{JC}$ is isomorphic to the affine space of dimension 
$$
\dim J_{\Delta}=\sum_{i,j} \left(|\Gaps(a_{j,i})|-|\Gaps(a_{j,i}+md)|\right).
$$
where $\Gaps(x)=[x,+\infty)\setminus \Delta$ and $a_{j,i}$ are the $(nd)$-generators of $\Delta$.
\end{theorem}

\begin{remark}
Although both  \cite{Piont} and \cite{CheP} use the notions of admissible subsets, there is an important difference between their definitions. In \cite{Piont}, as in this paper, the definition of admissible subsets is purely combinatorial. On the other hand, \cite{CheP} gives a geometric definition by  saying that $\Delta$ is admissible if  $J_{\Delta}$ is nonempty.  Lemma \ref{lem: admissible necessary} and Theorem \ref{thm: admissible intro} show that the two definitions agree for $s=1$.  However, we explain in the next subsection that for $s>1$ the combinatorial and geometric definitions of admissible subsets diverge, and we always use the combinatorial one in this paper. 
\end{remark}

To prove Theorem \ref{thm: admissible intro}, we need to analyze the explicit equations defining $J_{\Delta}$, which occupies most of Section \ref{sec: geometry}. The coordinates on $J_{\Delta}$ are indexed by $\Gaps(a_{j,i})$ and the equations are indexed by $\Gaps(a_{j,i}+md)$ for all $j,i$. We show in Proposition \ref{prop:adm-tria} that with an appropriately defined partial order on the coordinates, the equations can be written in such a way that they express certain coordinates in terms of smaller coordinates with respect to this order. This allows us to eliminate the variables one by one,  until all the equations are exhausted.

\begin{remark}
This approach  follows the ideas of \cite{Piont} and \cite{CheP}. In the nutshell, in both settings of \cite{Piont} and \cite{CheP} the linear parts of the equations are linearly independent and the nonlinear parts depend on smaller coordinates (with respect to the order), hence one can just keep track of the linear parts. \newline
The main new conceptual obstacle that we encounter in the proof of Theorem \ref{thm: admissible intro} is that for sufficiently large $m,n,d$ the {\bf linear parts of many equations become dependent} even if $\Delta$ is admissible, see Example \ref{ex: strata upd} and Lemma \ref{lem: suspicious dependent}. We develop combinatorial machinery to keep track of all such dependencies, and introduce some novel changes of coordinates to resolve these. The order on new variables is also changed. The change of variables heavily uses the fact that $C$ is generic and  $\lambda\neq 0$.
\end{remark}

Theorem  \ref{thm: admissible intro} immediately yields the Poincar\'e polynomial of $\overline{JC}$, described in terms of admissible subsets. Still, a significant combinatorial work is required to relate it to Dyck paths which appear in Theorem \ref{thm: intro main}. Following \cite{GMVDyck}, we introduce an equivalence relation on all $(nd,md)$-invariant subsets of $\Z_{\ge 0}$. The main result of \cite{GMVDyck} establishes a bijection between the equivalence classes and Dyck paths in $(nd)\times (md)$ rectangle. In this paper, we prove Theorem \ref{thm: unique admissible} which shows that in each equivalence class there exists a unique admissible subset. 
By combining these results, we establish a bijection between the admissible subsets and the Dyck paths, and complete the proof of Theorem \ref{thm: intro main}.

Finally, in Section \ref{sec: daha} we recall some results on rational shuffle theorem, elliptic Hall algebra and DAHA, and prove Theorem \ref{thm: intro daha}.

\subsection{Beyond generic curves}

The last section of the paper is focused on more general plane curve singularities 
$$(x(t),y(t))=\left(t^{nd},t^{md}+\lambda t^{md+s}+\ldots\right)$$ 
with two Puiseux pairs. We assume $\GCD(m,n)=\GCD(d,s)=1$ and $\lambda\neq 0$, the case $s=1$ corresponds to generic curves above. 

The definition of admissible subsets has a natural generalization which we call $s$-admissible (see Definition \ref{def: s admissible}). Combinatorially, we can parametrize the $s$-admissible subsets as following:

\begin{theorem}
\label{thm: intro s admissible}
There is a bijection between the $s$-admissible subsets and {\bf cabled Dyck paths} which are parametrized by the choice of a $(d,s)$ Dyck path $P$ with vertical runs $v_i(P)$ and a tuple of $(v_i(P)n,v_i(P)n)$-Dyck paths.
\end{theorem}

See Definition \ref{def: cabled Dyck} and Theorem \ref{thm: s admissible} for more details. Note that for $s=1$ there is only one $(d,s)$ Dyck path with $v_1(P)=d$, and hence cabled Dyck paths are simply $(dn,dm)$ Dyck paths as above.

On geometric side, we can still stratify the compactified Jacobian by the strata $J_{\Delta}$ parametrized by the $(nd,md)$-invariant subsets $\Delta$. However, the situation for $s>1$ is significantly more complicated than for $s=1$. In general, $J_{\Delta}$ might be nonempty if $\Delta$ is not $s$-admissible, and it might not be isomorphic to an affine space if it is $s$-admissible. Nevertheless, we conjecture the following weaker statement:

\begin{conjecture}
\label{intro conj: chi}
The Euler characteristic of the stratum $J_{\Delta}$ is given by
$$
\chi(J_{\Delta})=\begin{cases}
1 & \text{if}\ \Delta\ \text{is}\ s\text{-admissible},\\
0 & \text{if}\ \Delta\ \text{is not}\  s\text{-admissible}.
\end{cases}
$$
\end{conjecture} 

All existing computations by Piontkowski \cite{Piont} and Cherednik-Philipp \cite{CheP} are compatible with this conjecture. In particular, in examples one can find the strata $J_{\Delta}$ isomorphic to $\C^*\times \C^N$ (with Euler characteristic 0) and $\C^N\cup_{\C^{N-1}}\C^{N}$ (with Euler characteristic 1). Conjecture \ref{intro conj: chi} and Theorem \ref{thm: intro s admissible} immediately imply the following.

\begin{corollary}
\label{intro cor: chi}
Assuming Conjecture \ref{intro conj: chi}, the Euler characteristic of the compactified Jacobian equals the number of $s$-admissible subsets, and the number of cabled Dyck paths with parameters $(m,n),(d,s)$.
\end{corollary}

In \cite{KivTsai} Kivinen and Tsai  proved unconditionally that $\chi(\overline{JC})$ equals the number of cabled Dyck paths.

We plan to study the geometry of strata $J_{\Delta}$ in more details and generalize the above results to singularities with more than two Puiseux pairs in the future work. Another interesting future direction is understanding the Euler characteristics of Hilbert schemes of points on singular plane curves by similar methods, which would yield a new approach to the main conjecture of \cite{OS}. Finally, \cite{CheP} introduced a ``flagged" generalization of compactified Jacobians, and it would be interesting to see if these spaces associated to generic curves admit cell decompositions.





\section*{Acknowledgments}

We thank Francois Bergeron, Oscar Kivinen, Anton Mellit and Monica Vazirani for useful discussions. E. G. was partially supported by the NSF grant DMS-1760329, and by the NSF grant DMS-1928930 while E. G. was in residence at the Simons Laufer Mathematical Sciences Institute (previously known as MSRI) in Berkeley, California, during the Fall 2022 semester. M. M. was partially supported by the Simons Collaboration grant 524324. M. M. is also grateful to MSRI for hospitality during the Winter-Spring 2022, where the project started. A.O. was supported by the NSF grant DMS-1760373. 

\section{Background}
\label{sec: background}

\subsection{Compactified Jacobians and semigroups}
\label{sec: background Jacobian}

Let $C$ be an irreducible (and reduced) plane curve singularity at $(0,0)$. We can parametrize $C$ as $(x(t),y(t))$, and write the local ring of functions on $C$ as 
$\cO_C=\C[[x(t),y(t)]]\subset \C((t)).$ Given a function $f(t)\in \C((t))$, we can write 
$$
f(t)=\alpha_k t^{k}+\alpha_{k+1}t^{k+1}+\ldots,\ \alpha_k\neq 0
$$
and define the order of $f(t)$ by $\Ord f(t)=k$. It is well known that $\Ord$ is a valuation on $\cO_C$, that is
$$
\Ord(f\cdot g)=\Ord(f)+\Ord(g),\ \Ord(f+g)\ge \min(\Ord(f),\Ord(g)).
$$
The compactified Jacobian $\overline{JC}$ is defined as the moduli space of rank one  torsion free sheaves on $C$ or, equivalently, $\cO_C$-submodules $M\subset t^{-N}\C[[t]]$ (for sufficiently large $N$) such that 
\begin{equation}
\label{eq: M balance}
\dim \frac{t^{-N}\C[[t]]}{M}=\dim \frac{t^{-N}\C[[t]]}{\cO_C}.
\end{equation}
Note that $M$ is an $\cO_C$--submodule if and only if  $x(t)M\subset M,\ y(t)M\subset M$. Given such a subspace $M$, we define
$$
\Delta_M=\{\Ord\ f(t)\ :\ f(t)\in M\}\subset \Z.
$$
In particular, for $M=\cO_C$ we obtain the {\bf semigroup} of $C$:
$$
\Gamma=\Delta_{\cO_C}=\{\Ord\ f(t)\ :\ f(t)\in \cO_C\}.
$$
If $M$ is an $\cO_C$-submodule then $\Delta_M$ is a $\Gamma$-module: $\Delta_M+\Gamma\subset \Delta_M$. 
This motivates the following

\begin{definition}\label{def:Jac}
Given a subset $\Delta\subset \Z$, we define the stratum in the compactified Jacobian:
$$
J_{\Delta}:=\{M\subset \C((t)): \cO_CM\subset M,\ \Delta_M=\Delta\}.
$$ 
\end{definition}

\begin{definition}
A subset $\Delta\subset \Z$ is called balanced if $\Delta\subset [-N,+\infty)$ and $|[-N,+\infty)\setminus \Delta|=|[-N,+\infty)\setminus \Gamma|$ for sufficiently large $N$, and $0$-normalized if $\min \Delta=0$. 
\end{definition}

By definition, the subsets $\Delta_M$ corresponding to $M$ satisfying \eqref{eq: M balance} are balanced, but for many technical reasons it is easier to work with $0$-normalized subsets which correspond to $M\subset \C[[t]]$. The two classes of subsets are related by a shift, which corresponds to the multiplication of $M$ by a power of $t$. Clearly, the shift does not change $J_{\Delta}$ up to isomorphism, so in the rest of the paper we will abuse the notations and talk about the $J_{\Delta}$ for $0$-normalized subsets $\Delta$ instead.

Clearly, $J_{\Delta}$ give a stratification of $\overline{JC}$ and  $J_{\Delta}$ is empty if $\Delta$ is not $\Gamma$-invariant.  
As a warning to the reader, $J_{\Delta}$ could be empty even if $\Delta$ is $\Gamma$-invariant, as shown in the following example:

\begin{example}
\label{ex: admissible 4 6 7}
Consider the singularity $(x(t),y(t))=(t^4,t^6+t^7)$. It is easy to see that $y^2(t)-x^3(t)=2t^{13}+t^{14}$ has order $13$, and one can check that the semigroup $\Gamma$ in this case is generated by $4,6$ and $13$. The subset 
$$
\Delta=\{0,2,4,6,8,10,12,13,14,\ldots\}
$$
is clearly $\Gamma$-invariant. Suppose that 
$M\subset J_{\Delta}$, and $f_0=1+at+\ldots$ and $f_2=t^2+bt^3+\ldots$ are two elements of $M$ of orders $0$ and $2$ respectively. Then
$$
y(t)f_0-x(t)f_2=(t^6+t^7)(1+at+\ldots)-t^4(t^2+bt^3+\ldots)=(a+1-b)t^7+\ldots,
$$ 
$$
x^2(t)f_0-y(t)f_2=t^8(1+at+\ldots)-(t^6+t^7)(t^2+bt^3+\ldots)=(a-b-1)t^9+\ldots
$$
Observe that $a+1-b$ and $a-b-1$ cannot vanish at the same time, so either $M$ contains an element of order $7$ or an element of order 9, contradiction. Therefore the corresponding stratum $J_{\Delta}$ is empty.   
\end{example}

We will see  that a necessary condition for $J_{\Delta}$ being nonempty is that $\Delta$ satisfies the so-called admissibility property, which we define below.

\begin{remark} To replace base field \(\mathbb{C}\) by arbitrary field \(k\) we need a parametrization \(x(t),y(t)\)  of the germ of the curve singularity to be
    defined over  \(k\).
  \end{remark}
\subsection{Generic curves}

 It is well known that any plane curve singularity $C$ can be parametrized using Puiseux expansion:
$$
x=t^{nd},\ y=t^{md}+\lambda t^{md+1}+\ldots,\ \GCD(m,n)=1.
$$
It is convenient to assume $n<m$, but sometimes we will not need this assumption. If $d=1$ then $C$ has one Puiseux pair $(n,m)$ and its link is the $(n,m)$ torus knot. In this paper, we will be mostly interested in the case $d>1$.

\begin{definition}
Assume $d>1$. A plane curve singularity $C$ is called {\em generic} (or generic curve in short) if $\lambda \neq 0$. 
\end{definition}

It is easy to see that the definition of a generic curve is symmetric in $n$ and $m$. Indeed, we can choose the new parameter 
$$
\widetilde{t}=\sqrt[md]{y(t)}=t\sqrt[md]{1+\lambda t^{md+1}+\ldots}=t\left(1+\frac{\lambda}{md}t+\ldots\right),
$$
then 
$$
y=\widetilde{t}^{md},\ x=\widetilde{t}^{nd}-\frac{n\lambda }{md}\widetilde{t}^{nd+1}+\ldots.
$$
For a generic plane curve singularity, the difference $y^{n}-x^{m}$ has order $nmd+1$, and one can check that the semigroup $\Gamma$ is generated by $nd,md$ and $nmd+1$.

If $n=1$, then the singularity has one Puiseux pair $(d,md+1)$. Otherwise, it has two Puiseux pairs $(nd,md)$ and $(d,md+1)$, which completely determine the topological type of the corresponding knot as a $(d,mnd+1)$-cable of the $(n,m)$ torus knot.

\begin{remark}The  above argument requires the base field to be algebraically closed of characteristic \(0\). However, the symmetry is not used in the rest of the paper. 
  \end{remark}





\subsection{Invariant subsets}
\label{sec: comb gens}





A subset $\Delta\subset \Z_{\ge 0}$ is called $0$-normalized if $0\in \Delta$. We will mostly consider $0$-normalized subsets, as any subset of $\Z_{\ge 0}$ can be shifted to a unique $0$-normalized one. We call $\Delta$ cofinite if $\Z_{\ge 0}\setminus \Delta$ is finite. For a cofinite subset $\Delta$ and $x\ge 0$ we write 
$$
\Gaps(x)=\Gaps_{\Delta}(x):=[x,+\infty)\setminus \Delta.
$$

\begin{definition}
\label{def: generators}
We call $\Delta$ an $(nd)$-invariant subset if $nd+\Delta\subset \Delta$. A number $a$ is called an $(nd)$-generator of $\Delta$ if $a\in \Delta$ but $a-nd\notin \Delta$. 
\end{definition}

It is clear that for a cofinite $(nd)$-invariant subset $\Delta$ there is exactly one $(nd)$-generator in each remainder modulo $nd$. We will group them according to their remainders modulo $d$, so that the generators $a_{j,i},\ i=0,\ldots,n-1$ all have remainder $j$ modulo $d$ and different remainders modulo $nd$. 
We write 
$$
\Delta_{j}=\bigcup_{i=0,\ldots,n-1} (a_{j,i}+dn\mathbb{Z}_{\ge 0}),\ \Delta=\bigcup_{j=0,\ldots, d} \Delta_j.
$$

\begin{lemma}
\label{lem: comb syzygy}
Suppose that $\Delta$ is a cofinite $(nd)$-invariant subset with $(nd)$-generators $a_{j,i}$ as above. Then $\Delta$ is $(md)$-invariant if and only if one can cyclically order the generators $a_{j,i}$ for each $j$ such that  
\begin{equation}
\label{eq: comb syzygy}
a_{j,i}+md=a_{j,i+1}+\alpha_{j,i} nd,\quad \alpha_{j,i}\ge 0
\end{equation}
\end{lemma}

\begin{proof}
Since $\Delta$ is $(nd)$-invariant, it is $(md)$-invariant if and only if $a_{j,i}+md\in \Delta$. Since $a_{j,i}+md\equiv j\mod d$, this is equivalent to existence of a generator $a_{j,i'}\equiv a_{j,i}+md\mod nd$ such that $a_{j,i'}\le a_{j,i}+md$.

Observe that $(a_{j,i'}-j)/d \equiv (a_{j,i}-j)/d+m\mod n$. Since $m$ and $n$ are coprime, for a fixed $j$ the permutation $a_{j,i}\to a_{j,i'}$ is a single cycle, and the result follows.
\end{proof}
 
For a cofinite $(nd,md)$-invariant subset $\Delta,$ we will always assume that for each $j$ the generators $a_{j,i}$ are cyclically ordered as above. Furthermore, we will consider the second index $i$ modulo $n,$ so that $a_{j,i}=a_{j,i+n},$ and we will consider the first index $j$ modulo $m,$ so that $a_{j,i}=a_{j+m,i}.$ 

We will call the integers $a_{j,i}+md$  the combinatorial syzygys of $\Delta$. It is clear that  in each remainder modulo $d$ there are $n$ such syzygys, and the equations  \eqref{eq: comb syzygy} form an $n$-cycle.

\begin{example}
\label{ex: admissible 4 6 7 combi}
In Example \ref{ex: admissible 4 6 7} we have $d=2,n=2$ and $m=3$. The subset 
$$
\Delta=\{0,2,4,6,8,10,12,13,14,\ldots \}
$$
contains all integers starting from 12, so it is cofinite.  The $4$-generators are $0,2,13,15$ and the combinatorial syzygys
$$
0+6=2+4,\ 2+6=0+2\cdot 4,\ 13+6=15+4,\ 15+6=13+2\cdot 4.
$$
In the notations of Lemma \ref{lem: comb syzygy}, we can write
$$
a_{0,0}=0,\ a_{0,1}=2,\ a_{1,0}=13,\ a_{1,1}=15
$$
and
$$
\alpha_{0,0}=1,\ \alpha_{0,1}=2,\ \alpha_{1,0}=1,\ \alpha_{1,1}=2.
$$
\end{example}

\begin{example}
\label{ex: strata combi}
Suppose that $(nd,md)=(6,9)$ so that $d=3,n=2$ and $m=3$. The subset 
$$
\Delta=\{0,3,6,7,9,10,12,13,15,16,17,18,19,20,\ldots\}
$$
contains all integers starting from 15, so it is cofinite. It is not hard to see that it is $(6,9)$ invariant and its $6$-generators are $0,3,7,10,17,20$. The combinatorial syzygys are:
$$
0+9=3+6,\ 3+9=0+2\cdot 6,\ 7+9=10+6,\ 
$$
$$
10+9=7+2\cdot 6,\ 17+9=20+6,\ 20+9=17+2\cdot 6.
$$
In the notations of Lemma \ref{lem: comb syzygy}, we can write
$$
a_{0,0}=0,\ a_{0,1}=3,\ a_{1,0}=7,\ a_{1,1}=10,\ a_{2,0}=17,\ a_{2,1}=20.
$$
$$
\alpha_{0,0}=1,\ \alpha_{0,1}=2,\ \alpha_{1,0}=1,\ \alpha_{1,1}=2,\ \alpha_{2,0}=1,\ \alpha_{2,1}=2.
$$
Note that the first index corresponds to the remainder of the generators modulo $d=3$.
\end{example}

\begin{lemma}
\label{lem: gaps}
Assume that $\Delta$ is a cofinite $(nd,md)$-invariant subset. 
Suppose that $a_{j,i}+md=a_{j,i+1}+\alpha_{j,i} nd$ is a combinatorial syzygy, and $a_{j,i}+md+x$ is a gap in $\Delta$. Then both
$a_{j,i}+x$ and $a_{j,i+1}+x$ are gaps as well. 
\end{lemma}

\begin{proof}
If $a_{j,i}+x\in \Delta$ then $a_{j,i}+x+md\in \Delta$. If $a_{j,i+1}+x\in \Delta$ then $a_{j,i+1}+x+\alpha_{j,i}nd\in \Delta$ and $a_{j,i+1}+x+\alpha_{j,i} nd=a_{j,i}+x+md$. Contradiction. 
\end{proof}

\section{Topology of compactified Jacobians}
\label{sec: geometry}

Throughout this section we fix a generic plane curve singularity $C$ with parametrization $$(x(t),y(t))=(t^{nd},t^{md}+\lambda t^{md+1}+\ldots)$$ with $\lambda\neq 0$. We will use the notation $\cO_C=\C[[x(t),y(t)]]$ for the ring of functions on $C$.

\begin{remark}
The existence of such parametrization over arbitrary field is not automatic. However, if we assume this specific parametrization for $C$, all the arguments in this section are valid over any field of characteristic larger than $n$ (see Remarks \ref{rem: admissible any field} and \ref{rem: basis any field} for more details). 
  \end{remark}

We also fix a cofinite $(nd,md)$-invariant subset $\Delta\subset \Z_{\ge 0}$ with $(nd)$-generators $a_{j,i}$ as in Section \ref{sec: comb gens}.  We denote by $A=\{a_{j,i}\}$ the set of all $(nd)$-generators and follow  the notations \eqref{eq: comb syzygy} (see Lemma \ref{lem: comb syzygy}) in ordering them.

The main goal of this section is to describe the stratum $J_{\Delta}$ in the compactified Jacobian $\overline{JC}$.

\subsection{Equations for $J_\Delta$}
\label{sec:equations-j_delta}

Consider an $\cO_C$-module $M\in  J_{\Delta}.$ The following lemma is standard but we include the proof for the reader's convenience.

\begin{lemma}
\label{lem: canonical}
For all $k\in\Delta$ there exist a unique canonical representative
\begin{equation*}
f_k=t^k+\sum_{l\in\Gaps(k)} f_{k;l-k} t^l \in M.
\end{equation*}
The canonical generators form a topological basis in $M$.
\end{lemma}

\begin{proof}
Let us prove the existence first. For all $b\in \Delta$ we can choose some element $\widetilde{f}_b\in M$ of order $b$. Suppose that $\widetilde{f}_k=\sum_{i=0}^{\infty} \beta_i t^{k+i}$ and $\beta_0\neq 0$. By dividing by $\beta_0$, we can assume that $\beta_0=1$. If $\beta_i=0$ whenever $i>0$ and $k+i\in \Delta$, we are done since $\widetilde{f}_k$ is canonical. Otherwise we consider minimal $i>0$ such that $\beta_i\neq 0$ and $k+i\in \Delta$, and replace $\widetilde{f}_k$ by $\widetilde{f}_k-\beta_i\widetilde{f}_{a+i}$. By repeating this process, we construct a sequence of elements of $M$ which converges to a canonical element of order $k$ in $t$-adic topology. Since $M$ is closed in $t$-adic topology, we are done.

Now let us prove the uniqueness. Suppose that $f_k=t^k+\sum_{y\ge 0,k+y\notin \Delta}\alpha_{y}t^{k+y}$ and $f'_k=t^k+\sum_{y\ge 0,k+y\notin \Delta}\alpha'_{y}t^{k+y}$ are both contained in $M$. If $\alpha_{y}\neq \alpha'_{y}$ for some $y$, we get that $f_k-f'_k\neq 0$  and the order of $f_k-f'_k$ is a gap in $\Delta$, contradiction.
\end{proof}

Note that the representatives $\{f_a| a\in A\}$ generate the submodule $M$ over $\cO_C$. Here  $A=\{a_{j,i}\}$ the set of all $(nd)$-generators of $\Delta$ as above, and $f_a$ are defined as in Lemma \ref{lem: canonical}.
Vice versa, consider a collection of polynomials
\begin{equation}
\label{eq: generators}
\left\{g_a=t^a+\sum_{l\in\Gaps(a)} g_{a;l-a} t^l| a\in A\right\}.
\end{equation}
It will be convenient to consider the coefficients $g_{a;l-a}$ for $l\in\Gaps(a)$ as parameters. For $l\neq\Gaps(a)$ we use conventions
\[
g_{a;0}=1,\quad g_{a;l-a}=0 \mbox{ if } l\in \Delta\setminus\{a\}.
\]
Let $N$ be the $\cO_C$-submodule generated by the collection \eqref{eq: generators}, and let $\widetilde{N}$ be the $\C[[t^{dn}]]$ submodule generated by the same collection. Note that  $\widetilde{N}$ has a topological basis
\begin{equation}
\label{eq: basis N tilde}
\{g_{a}t^{\alpha nd}:\ a\in A,\ \alpha\ge 0\},\ \Ord(g_{a}t^{\alpha nd})=a+\alpha nd.
\end{equation}
Note that by definition of $(nd)$-generators (Definition \ref{def: generators}) the orders of polynomials in the basis \ref{eq: basis N tilde} are  pairwise distinct and the set of such orders coincides with $\Delta$.

Then $N\in J_{\Delta}$ if and only if $N=\widetilde{N}$ or, equivalently,
\begin{equation*}
y(t)g_a \in\widetilde{N}\ \forall\ a\in A.
\end{equation*}
For every $a\in A$ let $s_{a+dm}$ be a polynomial in $t$ whose coefficients are polynomials in \(g_{a';r}\), \(a'\in A\), \(r\in \mathbb{Z}_{>0}\)  such that
\begin{equation}
\label{eq: relations}
y(t)g_a-s_{a+dm}\in \widetilde{N}\quad \mathrm{and}\quad
  s_{a+dm}=\sum_{l\in\Gaps(a+dm)} s_{a+dm;l-dm-a} t^l.
\end{equation}
Such $s_{a+dm}$ is unique: given two distinct choices $s_{a+dm},s'_{a+dm}$ the difference $s_{a+dm}-s'_{a+dm}$ is contained in $\widetilde{N}$, but its order is a gap of $\Delta$ which is not possible.

Concretely, the polynomials $s_{a+dm}$ can be defined by reducing $y(t)g_a$ modulo $\widetilde{N}$ using the basis \eqref{eq: basis N tilde}, similarly to Lemma \ref{lem: canonical}.
For a recursive construction of \(s_{a+dm}\) see formulas \eqref{eq:s0} and \eqref{equation:iteration}. 
The condition $y(t)g_a\in\widetilde{N}$ is equivalent to $s_{a+dm}=0$ for all \(a\in A\).

We can summarize the above discussion as follows:
\begin{proposition}
The stratum \(J_\Delta\) is isomorphic to a subset of the affine space \(\Gen=\C^{G(\Delta)}\) defined by \(E(\Delta)\) equations:
\begin{equation}\label{eq:GE}
  G(\Delta)=\sum_{i,j}|\Gaps(a_{j,i})|,\quad E(\Delta)=\sum_{i,j}|\Gaps(a_{j,i}+dm)|.\end{equation}
In particular, the natural coordinates on \(\Gen\) are given by coefficients \(g_{a_{j,i};x}\) for \(a_{i,j}+x\in \Gaps(a_{j,i})\)  in generators \eqref{eq: generators}, there is a tuple of  coordinates for each combinatorial generator $a_{i,j}$ of $\Delta$. 
The equations are given by 
\[s_{a_{j,i}+dm;x}(g)=0, \quad a_{j,i}+dm+x\notin\Delta.\]
There is a tuple of  relations for each combinatorial syzygy of $\Delta$ as in Lemma \ref{lem: comb syzygy}.
\end{proposition} 
\begin{example}
We can revisit Example \ref{ex: admissible 4 6 7} and Example \ref{ex: admissible 4 6 7 combi} with new notations. We have
$g_0=1+at+\ldots$ and $g_2=t^2+bt^3+\ldots$, so in particular $g_{0;1}=a$ and $g_{2;1}=b$. Using the computations in Example \ref{ex: admissible 4 6 7}, we observe that 
$$
s_{6;1}=a+1-b,\ s_{8;1}=b+1-a.
$$
In total, there are $\Gaps(0)+\Gaps(2)+\Gaps(13)+\Gaps(15)=6+5+0+0=11$ coordinates and $\Gaps(6)+\Gaps(8)+\Gaps(19)+\Gaps(21)=3+2+0+0=5$ equations, but 
the equations $s_{6;1}=s_{8;1}=0$ are already contradictory.
\end{example}

\begin{example}
\label{ex: strata upd}
As in Example \ref{ex: strata combi}, consider $(nd,md)=(6,9)$ and the module
$$
\Delta=\{0,3,6,7,9,10,12,13,15,16,17,18,19,20,\ldots\}
$$
with $6$-generators $0,3,7,10,17,20$. We have 
$$
g_0=1+a_1t+a_2t^2+\ldots,\ g_3=t^3+b_1t^4+b_2t^5,\ g_7=t^7+c_1t^8+\ldots,\ g_{10}=t^{10}+d_1t^{11}+\ldots
$$
Here $a_i=g_{0;i},b_i=g_{3;i},c_i=g_{7;i}$ and $d_i=g_{10;i}$ in the above notations.
Note that $g_{17}=t^{17}$ and $g_{20}=t^{20}$ since $\Gaps(17)=\Gaps(20)=0$.
The only interesting syzygys are of degrees $9$ and $12$:
$$
(t^9+\lambda t^{10})g_0-t^6g_3-(a_1+\lambda-b_1)g_{10}=(a_2+\lambda a_1-b_2-(a_1+\lambda-b_1)d_1)t^{11}+\ldots
$$
and
$$
(t^9+\lambda t^{10})g_3-t^{12}g_0-(b_1+\lambda-a_1)t^6g_{7}=(b_2+\lambda b_1-a_2-(b_1+\lambda-a_1)c_1)t^{14}+\ldots
$$
so we get degree 2 equations
$$
s_{9;2}=a_2+\lambda a_1-b_2-(a_1+\lambda-b_1)d_1=0,\ s_{12;2}=b_2+\lambda b_1-a_2-(b_1+\lambda-a_1)c_1=0.
$$
Note that the linear parts of these equations are $a_2-b_2$ and $b_2-a_2$ respectively, which are linearly dependent. If we add the equations, these cancel out and we get
$$
\lambda (a_1+b_1)-(a_1+\lambda-b_1)d_1-(b_1+\lambda-a_1)c_1=0.
$$
For $\lambda=0$, we get a nonlinear equation $(a_1-b_1)(c_1-d_1)=0$. For $\lambda\neq 0$, however, we can solve the equations as follows: first, fix arbitrary $b_1-a_1=\theta$, then
$$
a_1+b_1=\frac{1}{\lambda}((\theta+\lambda)d_1-(\lambda-\theta)c_1)
$$
and this determines $a_1$ and $b_1$ uniquely. Given $a_1$ and $b_1$, we can determine $a_2-b_2$ as well.
There is one more equation $s_{9;5}=0$ of degree 5, which is easy to solve. Overall, we get an affine space of dimension
$$
\Gaps(0)+\Gaps(3)+\Gaps(7)+\Gaps(10)-\Gaps(9)-\Gaps(12)=7+5+3+2-2-1=14
$$
as predicted by \cite[Section 4.4]{CheP}.
\end{example}

\subsection{Admissibility}
\label{sec:comb-ag}
In this subsection we discuss an admissibility condition on \(\Delta\) generalizing Example \ref{ex: admissible 4 6 7}.  We show that
if \(\Delta\) is not admissible then \(J_\Delta=\emptyset\). In subsequent section we
describe \(J_\Delta\) for admissible \(\Delta\)'s.

We will consider $(md,nd)$-invariant subsets, and say that such a subset $B$ is supported in a remainder $r$ if all elements of $B$ are congruent to $r\mod d$.

\begin{lemma}
\label{lem: AB}
Suppose that $B,C$ are $(md,nd)$-invariant subsets supported in the same remainder $r\mod d$, and $b_i,c_i$ are their $(nd)$-generators. Then $B\subset C$ if and only if $c_i-nd\notin B$ for all $i$.
\end{lemma}

\begin{proof}
There exists a permutation $\sigma$ such that $c_i\equiv b_{\sigma(i)}\mod nd$ for all $j$. We have 
$$
B\subset C\Leftrightarrow \forall i\ c_i\le b_{\sigma(i)}\Leftrightarrow \forall i\ c_i-nd<b_{\sigma(i)}\Leftrightarrow \forall i\ c_i-nd\notin B.\qedhere
$$
\end{proof}


\begin{definition}
\label{def: suspicious}
Given an $(nd,md)$-invariant subset $\Delta$ and a remainder $j$ modulo $d$,
we call $x>0$ $j$-suspicious  if $a_{j,i}+md+x\notin \Delta$ for all $i$, and suspicious if it is $j$-suspicious for at least one remainder $j$.
\end{definition}

\begin{example}
In Example \ref{ex: admissible 4 6 7 combi} one can check that $x=1$ is $0$-suspicious since $0+6+1=7$ and $2+6+1=9$ are both not in $\Delta$. In Example \ref{ex: strata combi}  one can check that $x=2$ is $0$-suspicious since $0+9+2=11$ and $3+9+2=14$ are both not in $\Delta$. 
\end{example}
 
Note that in both examples suspicious numbers correspond to dependencies between the linear parts of the equations. More generally, we have the following result.

\begin{lemma}
\label{lem: suspicious dependent}
A number $x$ is $j$-suspicious if and only if the linear parts of the equations $s_{a_{j,i}+md;x}$ are linearly dependent for $i=0,\ldots,n-1$. 
\end{lemma}

Here by linear parts we mean linear terms in $g_{a;x}$ not containing $\lambda$ or higher coefficients of $y(t)$.

\begin{proof}
Recall that by Lemma \ref{lem: comb syzygy} we have $a_{j,i}+md=a_{j,i+1}+\alpha_{j,i}nd$. We can write 
$$
s_{a_{j,i}+md}(t)=y(t)g_{a_{j,i}}(t)-(t^{nd})^{\alpha_{j,i}}g_{a_{j,i+1}}-\ldots
$$
where $\ldots$ denotes certain correction terms with nonlinear coefficients, see  \eqref{eq:s0} and \eqref{equation:iteration} below.
The coefficient at $t^{a_{j,i}+md+x}$ in  $y(t)g_{a_{j,i}}(t)$ equals $g_{a_{j,i};x}+\lambda g_{a_{j,i};x-1}+\ldots$, while the coefficient at $(t^{nd})^{\alpha_{j,i}}g_{a_{j,i+1}}$ equals $g_{a_{j;i+1};x}$. 

Therefore the linear part of $s_{a_{j,i}+md;x}$ equals $g_{a_{j,i};x}-g_{a_{j,i+1};x}$, as long as $a_{j,i}+md+x\notin \Delta$ (otherwise there is no equation), $a_{j,i}+x\notin \Delta$ and $a_{j,i+1}+x\notin \Delta$ (otherwise the corresponding coordinates vanish). Note that by Lemma \ref{lem: gaps} $a_{j,i}+md+x\notin \Delta$ implies that $a_{j,i}+x\notin \Delta$ and $a_{j,i+1}+x\notin \Delta$. 
Now $x$ is $j$-suspicious if and only if all such equations  $s_{a_{j,i}+md;x}$ are present, and their linear parts $g_{a_{j,i};x}-g_{a_{j;i+1};x}$ sum up to zero. Otherwise these equations are linearly independent.
\end{proof}


\begin{lemma}\label{lem:susp}
A number $x$ is $j$-suspicious  if and only if $\Delta_{j+x}\subset \Delta_j+md+nd+x$.
\end{lemma}

\begin{proof}
The subset $\Delta_j+md+nd+x$ is supported in remainder $j+x$ and has $(nd)$-generators $a_{j,i}+md+nd+x$. The result now follows from Lemma \ref{lem: AB}.
\end{proof}

\begin{lemma}
\label{lem: mins}
Suppose that $\min \Delta_{j+x}\le\min \Delta_{j}$. Then $x$ is not $j$-suspicious.
\end{lemma}

\begin{proof}
Denote $b=\min \Delta_{j+x}$. We have $b\equiv j+x\mod d$, so $b-x\equiv j\mod d$ and there exists $i$ such that $b-x-md\equiv a_{j,i}\mod nd$. Now
$$
a_{j,i}+md+x>a_{j,i}\ge \min \Delta_{j}\ge b,\ a_{j,i}+md+x\equiv b\mod nd,
$$
hence $a_{j,i}+md+x\in \Delta_{j+x}$. Therefore $x$ is not $j$-suspicious.
\end{proof}

\begin{corollary}
\label{cor: last ok}
For any $0$-normalized $\Delta$ and $x\in \mathbb{Z}_{>0},$ $x$ is not $(d-x)$-suspicious. 
\end{corollary}

\begin{proof}
Note that by Lemma  \ref{lem: mins} if $d\mid x$ then $x$ is not suspicious. Suppose that $d\nmid x.$ Since $\Delta$ is $0$-normalized we have $\min \Delta_0=0<\min \Delta_{d-x}$, and Lemma \ref{lem: mins} applies.
\end{proof}

The following is a key definition of this section.

\begin{definition}
\label{def: admissible}
We call $\Delta$ admissible, if $1$ is not suspicious for $\Delta$.
\end{definition}

\begin{lemma}\label{lem:prev}
Suppose that $\Delta$ is admissible and $x$ is $j$-suspicious. Then there exists $i$ such that $a_{j,i}+x-1\notin \Delta.$
\end{lemma}

\begin{proof}
Assume that $a_{j,i}+x-1\in \Delta$ for all $i$, then $\Delta_j+x-1\subset \Delta_{j+x-1}$ and $\Delta_j+md+nd+x\subset \Delta_{j+x-1}+md+nd+1$. Since $x$ is $j$-suspicious, we have 
$\Delta_{j+x}\subset \Delta_j+md+nd+x\subset \Delta_{j+x-1}+md+nd+1$.
Therefore $1$ is suspicious for the remainder $j+x-1$, contradiction.
\end{proof}

\begin{lemma}
If $\Delta$ is admissible then $\Delta$ is $(mnd+1)$-invariant.
\end{lemma}

\begin{proof}
Suppose that $a\in \Delta$ has remainder $j\mod d$. Since $\Delta$ is admissible, there exists an integer $i$ such that $a_{j,i}+md+1\in \Delta$. There exists some $0\le s\le n-1$ such that $a+mds\equiv a_{j,i}\mod nd$. Since $\Delta$ is $(md)$-invariant, we get $a+mds=a_{j,i}+\alpha nd$ for some $\alpha\ge 0$. 
Now 
$$
a+mdn+1=(a+md(s+1)+1)+md(n-1-s)=(a_{j,i}+md+1)+\alpha nd+md(n-1-s)\in \Delta.
$$
\end{proof}

\begin{lemma}
\label{lem: admissible necessary}
Suppose that $J_\Delta\neq\emptyset.$ Then $\Delta$ is admissible.
\end{lemma}

\begin{proof}
We follow the strategy of Lemma \ref{lem: suspicious dependent}.
Suppose $\Delta$ is not admissible, then $1$ is suspicious for $\Delta$ and there exists some $j$ such that $a_{j,i}+md+1\notin \Delta$ for all $i$. This implies $a_{j,i}+1\notin \Delta$ for all $i$, so we have canonical generators  $g_{a_{j,i}}=t^{a_{j,i}}+g_{a_{j,i};1}t^{a_{j,i}+1}+\ldots$ 
Recall that by  Lemma \ref{lem: comb syzygy}  we get $a_{j,i}+md=a_{j,i+1}+\alpha_{j,i} nd$ and we get
$$
t^{\alpha_{j,i} nd}g_{a_{j,i+1}}-y(t)g_{a_{j,i}}=t^{\alpha_{j,i} nd}g_{a_{j,i+1}}-(t^{md}+\lambda t^{md+1}+\ldots)g_{a_{j,i}}=(g_{a_{j,i+1};1}-g_{a_{j,i};1}-\lambda)t^{a_{j,i}+md+1}+\ldots
$$
hence
$g_{a_{j,i+1};1}-g_{a_{j,i};1}-\lambda=0$ for all $i$.  By adding all these together we get $n\lambda=0$, contradiction.
\end{proof}

\begin{remark}
\label{rem: admissible any field}
  The very last part of the argument the lemma relies on the fact that we work in
  characteristic \(0\). More generally, the argument is valid is we assume that the characteristic of the base field is
  coprime with \(n\).
\end{remark}

See Example \ref{ex: admissible 4 6 7} for an illustration of this lemma.

\begin{example}
If $d=1$ then all $(n,m)$-invariant modules are admissible. Indeed, if $\Delta$ is $(n,m)$-invariant with $n$-generators $a_1<\ldots<a_n$ then $[a_n,+\infty)\subset \Delta$. In particular, for any $x>0$ one has $a_n+m+x\in \Delta$, so $x$ is not suspicious.
\end{example}

\begin{example}
Suppose that $n=1$ and $\Delta$ is $(d,md)$-invariant (that is, $d$-invariant). Suppose that $a_{j,0}$ are $d$-generators of $\Delta$. Then $1$ is $j$-suspicious if and only if $a_{j,0}+md+1\notin \Delta$, so $\Delta$ is admissible if and only if $a_{j,0}+md+1\in \Delta$ for all $j$, which is equivalent to $\Delta$ being $(md+1)$-invariant. 
\end{example}

\subsection{Paving by affine spaces}
\label{sec:paving-affine-spaces}

Recall that  \(\Gen\) has coordinates $g_{a_{j,i};x}$ where $a_{j,i}+x\notin \Delta$. Sometimes we will use the notation $g_{a_{j,i};x}$ for all $x$, assuming
\begin{equation}
\label{eq: g zero}
g_{a_{j,i};x}=0\quad \text{if}\ a_{j,i}+x\in \Delta.
\end{equation}

It is also convenient to consider $\Gen$ as a graded vector space
\begin{equation*}
\Gen=\bigoplus_{x=1}^\infty \Gen_x,
\end{equation*}
where $\Gen_x$ is spanned by $g_{a_{j,i};x}$ with a fixed $x.$

There are coordinates on \(\Gen\) that are most suitable for study of equations \(s_k\) in the case \(\Delta\) is admissible.  
Let us define
\[g^{-}_{j,i;x}=g_{a_{j,i};x}-g_{a_{j,i+1};x}, \quad i=0,\dots,n-1.\] These functions  are linearly dependent, for example
  \(\sum_i g^-_{j,i;x}=0\). We choose a  subset of these  as follows:
\begin{itemize}
\item[(1)] If  \(x\) is \(j\)-suspicious, we define \(I(j;x)=\{0,\dots,n-2\}\). 

\item[(2)] If \(x\) is not \(j\)-suspicious, we can define \(I(j;x)\) to be a set of \(i\) such that \(a_{j,i}+dm+x\notin \Delta\). 
\end{itemize}

\begin{lemma}
\label{lem: linear independent}
The functions \(g^-_{j,i;x}\), \(i\in I(j;x)\) are linearly independent.
\end{lemma}

\begin{proof}
Since $\Delta$ is $(dm)$-invariant, the condition $a_{j,i}+dm+x\notin \Delta$ implies $a_{j,i}+x\notin \Delta$. We consider two cases:

1) If $x$ is $j$-suspicious then, by definition, $a_{j,i}+md+x\notin \Delta$ and hence $a_{j,i}+x\notin \Delta$ for all $i$. In particular, \eqref{eq: g zero} does not apply and all $g_{a_{j,i};x}$ are linearly independent coordinates on $\Gen$. Therefore
\(g^-_{j,i;x}\), \(i\in \{0,\dots,n-2\}\) are linearly independent as well. 

2) Suppose that $x$ is not $j$-suspicious. Then for some $s$ we have \(a_{j,s}+dm+x\in \Delta\) and $s\notin I(j;x)$, therefore $I(j;x)$ is a proper subset of $\{0,\ldots,n-1\}$. Furthermore, for $i\in I(j;x)$ the condition \eqref{eq: g zero} does not apply, so
$g_{a_{j,i};x}$ for $i\in I(j;x)$ form a linearly independent subset of coordinates on $\Gen$. 

Assume that there is a nontrivial linear relation between $g^{-}_{j,i';x}=g_{a_{j,i'};x}-g_{a_{j,i'+1};x}$ which contains $g^{-}_{j,i;x}$ for some $i\in I(j;x)$. Since $g_{a_{j,i};x}$ is a coordinate on $\Gen$ that cancels out, the relation must contain $g^{-}_{j,i-1;x}$ as well, so $i-1\in I(j;x)$. By induction, we conclude that $i\in I(j;x)$ for all $i$, contradiction. Therefore \(g^-_{j,i;x}\), \(i\in I(j;x)\) are linearly independent.
\end{proof}



  
  Finally, if there exists at least one integer $i$ such that $a_{j,i}+x\notin \Delta,$ then we set \[g^+_{j;x}=\sum_{i:a_{j,i}+x\notin \Delta}  g_{a_{j,i};x}.\] In particular, if \(x+1\) is suspicious then Lemma~\ref{lem:prev} guarantees that the sum is not empty.

  Thus if $I(j;x)\neq\emptyset$ then \(\{g^-_{j,i;x},g^+_{j;x}\}\), \(i\in I(j;x)\) are linearly independent linear
  coordinates on the space of generators \(\Gen_x\). For each \(j,x\) such that $I(j;x)\neq\emptyset$ let us fix a subset
  \(\bar{I}(j;x)\) such that \(\left\{g^-_{j,i;x},g^+_{j;x}, g_{a_{j,i'};x}\right\}\) \(i\in I(j;x)\),
  \(i'\in \bar{I}(j;x)\)  is a basis of linear coordinates on \(\Gen_x\). For $I(j;x)=\emptyset$ set $\bar{I}(j;x)=\{0,1,\dots,n-1\}.$

  \begin{remark}
\label{rem: basis any field}
To show that \(\left\{g^-_{j,i;x},g^+_{j;x}, g_{a_{j,i'};x}\right\}\) \(i\in I(j;x)\),
      \(i'\in \bar{I}(j;x)\) is a basis of linear coordinates we used that we work over \(\mathbb{C}\). The argument also works for a base field of characteristic larger than \(n\). 
    The rest of the arguments in this section are also valid under this assumption.
    \end{remark}
  
We will need to briefly use the higher coefficients in the expansion of $y(t),$ so let 

\begin{equation*}
y(t)=t^{md}+\lambda t^{md+1}+\lambda_2 t^{md+2}+\lambda_3 t^{md+3}+\ldots. 
\end{equation*}

It will be convenient to think of the coefficients $\lambda,\lambda_2,\ldots$ as parameters.  

As we will see later the coordinates \(g_{a_{j,i'};x}\), \(i'\in \bar{I}(j;x)\) are not constrained by the equations for \(J_\Delta\). Thus let us set notation for the polynomial ring of these free variables and of \(\lambda\):
  \[R_{\free}=\C[\lambda,\lambda^{-1},\lambda_2,\lambda_3,\ldots][g_{a_{j,i'};x}]_{i'\in \bar{I}(j;x)}.\]

  On the rest of the coordinates  \(g^*_{*;*}\) we introduce a partial order generated as follows. Allowing $*$ to take any independent values,
  \[
  g^-_{*,*;x}<g^+_{*;x}<g^-_{*,*,x+1},
  \] 
coordinates $g^-_{*,*;x}$ are ordered in any arbitrary way, and

$$
g^+_{j+1;x}<g^+_{j;x} 
$$ 
(cyclic notation modulo $d$) for $j\neq d-x-1.$ 

The partial order is compatible with the grading on the algebra \(\mathbb{C}[\Gen\times \C_\lambda]\) defined by
  \[\deg_+(g_{a;x})=x, \quad \deg_+(\lambda)=1, \quad \deg_+(\lambda_i)=i.\] 

  Let us analyze graded properties of the equations  for \(J_\Delta\). We fix \(\deg_g\) for the usual degree 
  \[\deg_g(g_{a;x})=1,\quad \deg_g(\lambda)=\deg_g(\lambda_i)=0.\]
  
  \begin{lemma}\label{lemma:expansionl} For \(x\) such that \(a_{ji}+dm+x\notin \Delta\)
    polynomial \(s_{a_{j,i}+dm;x}\) is homogeneous with respect to the grading \(\deg_+\)
    and
\[\deg_+(s_{a_{j,i}+dm;x})=x.\]

Furthermore, for any \(j,i,x\):
\begin{equation}\label{equation:gjix-}
s_{a_{j,i}+dm;x}=g_{j,i;x}^- + \mbox{polynomial in variables} < g^-_{j,i;x} \mbox{ with coefficients in } R_{\free},
\end{equation}

and if \(x\) is \(j\)-suspicious then
\begin{equation}\label{equation:gjx+}
\sum_{i=0}^{n-1}s_{a_{j,i}+dm;x}=\lambda g^+_{j;x-1}+\mbox{polynomial in variables} < g^+_{j;x-1} \mbox{ with coefficients in } R_{\free}.
\end{equation}      
      
    \end{lemma}








    \begin{proof}
      The expansion \(s_{a_{j,i}+dm}(t)\)  is computed by an iterated process. We define a sequence of power series
$$
s^{(k)}_{a_{j,i}+dm}=\sum_{\ell=0}^{\infty} s^{(k)}_{a_{j,i}+dm;l}t^{a_{j,i}+dm+l}
$$
such that 
 \(s_{a_{j,i}+dm}=s^{(\infty)}_{a_{j,i+dm}}\). Here the seed of the process is defined by

 \begin{equation}\label{eq:s0}
        s^{(0)}_{a_{j,i}+dm}=g_{a_{j,i}}y(t)-t^{dn\alpha_{j,i}}g_{a_{j,i+1}}
\end{equation}
where $\alpha_{j,i}$ are defined by \eqref{eq: comb syzygy}. The step of the process if defined as follows. 
We set
\begin{equation}\label{equation:iteration}
s^{(k+1)}_{a_{j,i}+dm}=s^{(k)}_{a_{j,i}+dm}-\sum_{x>0:a_{j,i}+dm+x\in \Delta} s^{(k)}_{a_{j,i}+dm;x}g_{a_{j+x,i'}}t^{ndu},
\end{equation}
where for every summand $i'$ and $u$ are chosen so that \(a_{j,i}+dm+x=a_{j+x,i'}+ndu\). Note that on each step the term with the smallest $x$ such that $s^{(k)}_{a_{j,i}+dm;x}\neq 0$ and $a_{j,i}+dm+x\in\Delta$ cancels out. Therefore, the process converges. Since \(s^{(0)}_{a_{j,i}+dm;x}\) is homogeneous of degree $x$ with respect to \(\deg_+\)  and the iteration process preserves homogeneity, the first statement follows.

One immediately concludes by induction that for every $k\in\mathbb{Z}_{\ge 0}$ and every $x\in\mathbb{Z}_{\ge 0}$ such that $a_{j,i}+dm+x\in\Delta$ one has $\deg_g(s^{(k)}_{a_{j,i}+dm;x})>k.$ Furthermore, any monomial of $\deg_+=x$ and $\deg_g>2$ can only depend on $g_{a_{*,*};z}$ with $z<x-1.$  Therefore, for $k>0$ and any $x$ one has

\begin{equation}\label{equation:s(k) vs s(k+1)}
s^{(k)}_{a_{j,i}+dm;x}=s^{(k+1)}_{a_{j,i}+dm;x}+\ldots,
\end{equation}
where $\ldots$ is a series depending only on $g_{a_{*,*};z}$ with $z<x-1.$ In particular
\begin{equation}\label{equation:s vs s(1)}
s_{a_{j,i}+dm;x}=s^{(1)}_{a_{j,i}+dm;x}+\ldots.
\end{equation}
Here and below we use the same convention about $\ldots$ as in \eqref{equation:s(k) vs s(k+1)}.

Furthermore, from \eqref{eq:s0} we get
\begin{equation*}
s^{(0)}_{a_{j,i}+dm;x}=g_{a_{j,i};x}-g_{a_{j,i+1};x}+\lambda g_{a_{j,i};x-1}+\ldots=g^-_{j,i,x}+\lambda g_{a_{j,i};x-1}+\ldots.
\end{equation*}
Using \eqref{equation:iteration}, we get
\begin{multline}\label{equation:s(1)}
s^{(1)}_{a_{j,i}+dm;x}=s^{(0)}_{a_{j,i}+dm;x}-\sum_{0<y<x:a_{j,i}+dm+y\in \Delta} s^{(0)}_{a_{j,i}+dm;y}g_{a_{j+y,i'};x-y}\\
=g^-_{j,i,x}+\lambda g_{a_{j,i};x-1}+\sum_{0<y<x:a_{j,i}+dm+y\in \Delta} (g^-_{j,i;y}+\ldots)g_{a_{j+y,i'};x-y}+\ldots\\
%
\end{multline} 

Combining \eqref{equation:s vs s(1)} and \eqref{equation:s(1)} we immediately get Equation \eqref{equation:gjix-}. To get Equation \eqref{equation:gjx+} we take the sum:
\begin{multline*}
\sum_{i=0}^{n-1}s_{a_{j,i}+dm;x}=\sum_{i=0}^{n-1} s^{(1)}_{a_{j,i}+dm;x}+\ldots=\lambda g^+_{j;x-1}+\sum_{i=0}^{n-1}\sum_{0<y<x:a_{j,i}+dm+y\in \Delta} (g^-_{j,i;y}+\ldots)g_{a_{j+y,i'};x-y}+\ldots\\
\end{multline*}
Note that only terms with $y=1$ and $y=x-1$ contribute, other terms depend on $g_{a_{j,i};z}$ with $z<x-1$ only. Furthermore, for $y=x-1$ we have $g^-_{j,*;x-1}<g^+_{j,x-1}.$ For $y=1$ the variable $g_{a_{j+1,i'};x-1}$ is a linear combination of $g^-_{j+1,*;x-1}$ and $g^+_{j+1;x-1}$ (modulo variables from $R_{free}$). Since $x$ is $j$-suspicious, according to Corollary \ref{cor: last ok} we have $j\neq d-x=d-(x-1)-1,$ therefore $g^+_{j+1;x-1}<g^+_{j;x-1}.$ Hence we obtain Equation \eqref{equation:gjx+}.





    \end{proof}

\begin{proposition}\label{prop:adm-tria}
  If \(\Delta\) is admissible then
  \[J_{\Delta}=\mathbb{C}^{\dim(\Delta)},\quad \dim(\Delta)=G(\Delta)-E(\Delta),\]
where \(G(\Delta)\) and \(E(\Delta)\) are given by the equations \eqref{eq:GE}.
\end{proposition}
\begin{proof}
Recall that $J_\Delta$ is defined by the equations $s_{a_{j,i}+dm;x}=0$ for $j,i,x$ such that $a_{j,i}+dm+x\notin\Delta.$ We can modify this system of equations as follows. Whenever $x$ is $j$-suspicious, replace $s_{a_{j,n-1}+dm;x}=0$ by $\sum\limits_{i=0}^{n-1} s_{a_{j,i}+dm;x}=0.$ Clearly, the new system of equations is equivalent to the old one. Furthermore, according to Lemma \ref{lemma:expansionl}, the new system of equations expresses some of the elements of a basis of the space $\Gen$ in terms of the smaller variables with respect to the order $<.$ Therefore, one can use the equations to eliminate these variables one by one. Since $\dim\Gen=G(\Delta)$ and there are $E(\Delta)$ equations, we obtain the required result.

\end{proof}  

\section{Combinatorics}
\label{sec: combi}

\subsection{More on invariant subsets}
 
Let $\Delta$ be an $(nd,md)$-invariant subset. We call $b$ an $(md)$-cogenerator for $\Delta$ if $b\notin \Delta$ but $b+md\in \Delta$.

\begin{lemma}
\label{lem: dim gen cogen}
Let $\Delta$ be an admissible cofinite $(nd,md)$-invariant subset. The dimension of $J_{\Delta}$ equals to the number of pairs $(a,b)$ such that $a$ is an $(nd)$-generator of $\Delta$, $b$ is an $(md)$-cogenerator and $a<b$.
\end{lemma}

\begin{proof}
By Proposition  \ref{prop:adm-tria} the dimension of $J_{\Delta}$ equals $\sum_{a\in A} (|\Gaps(a)|-|\Gaps(a+md)|),$ where the sum is over $(nd)$-generators. The difference $|\Gaps(a)|-|\Gaps(a+md)|$ equals the number of $(md)$-cogenerators greater than $a$. 
\end{proof}

\subsection{Relatively prime case} Let $\Theta$ be an $n,m$-invariant subset in $\mathbb{Z}.$ As before, $n,m$ are relatively prime. 

\begin{definition}
The {\em skeleton} $S$ of $\Theta$ is the union of the $n$-generators and $m$-cogenerators of $\Theta$.
\end{definition}

An important way of visualizing $(n,m)$-invariant subsets is by using the \textit{periodic lattice paths} as follows. Consider a two-dimensional square lattice with the boxes labeled by integers via the linear function $l(x,y)=-nx-my$ (here $(x,y)$ are the coordinates of the, say, northeast corner of the square box). Since $n$ and $m$ are relatively prime, every integer appears exactly once in each horizontal strip of width $n$ (and in each vertical strip of width $m$). Furthermore, the labeling is periodic with period $(m,-n).$ Given an $(n,m)$-invariant subset $\Theta\subset\mathbb{Z},$ the corresponding periodic lattice path $P(\Theta)$ is the boundary that separates the set of boxes labeled by integers from $\Theta$ from the set of boxes labeled by integers from the complement $\mathbb{Z}\setminus\Theta.$ Note that the $n$-generators of $\Theta$ are exactly the labels in the boxes just south of the path, and the $m$-cogenerators of $\Theta$ are exactly the labels in the boxes just to the east of the path. In particular, a number $k$ belongs to the skeleton of an $(n,m)$-invariant subset if an only if the northwest corners of the boxes labeled by $k$ lie on the corresponding periodic path. We immediately get the following:

\begin{lemma}\label{lemma: skeleton vs periodic path}
Skeletons of two $(n,m)$-invariant subsets intersect if and only if the corresponding periodic lattice paths intersect.
\end{lemma}

See Figure \ref{Figure:periodic path} for an example of a periodic path corresponding to a $(3,5)$-invariant subset.

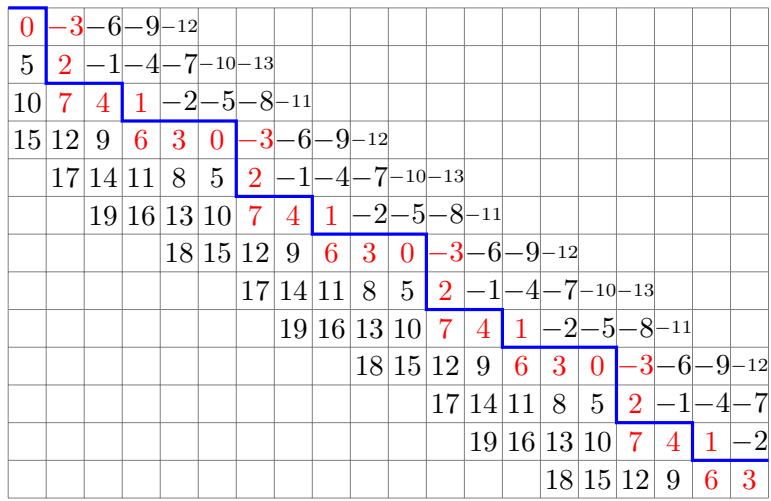
\begin{figure}
\begin{tikzpicture}[scale=0.5]
\draw  [very thin, gray](0,7) grid (20,20);
\draw [red] (0.5,19.5) node {$0$}; 
\draw (0.5,18.5) node {$5$}; 
\draw (0.5,17.5) node {$10$}; 
\draw (0.5,16.5) node {$15$}; 

\draw [red] (1.5,19.5) node {$-3$}; 
\draw [red] (1.5,18.5) node {$2$}; 
\draw [red] (1.5,17.5) node {$7$}; 
\draw (1.5,16.5) node {$12$}; 
\draw (1.5,15.5) node {$17$}; 

\draw (2.5,19.5) node {$-6$}; 
\draw (2.5,18.5) node {$-1$}; 
\draw [red] (2.5,17.5) node {$4$}; 
\draw (2.5,16.5) node {$9$}; 
\draw (2.5,15.5) node {$14$}; 
\draw (2.5,14.5) node {$19$}; 

\draw (3.5,19.5) node {$-9$}; 
\draw (3.5,18.5) node {$-4$}; 
\draw [red] (3.5,17.5) node {$1$}; 
\draw [red] (3.5,16.5) node {$6$}; 
\draw (3.5,15.5) node {$11$}; 
\draw (3.5,14.5) node {$16$}; 

\draw (4.5,19.5) node {\tiny $-12$}; 
\draw (4.5,18.5) node {$-7$}; 
\draw (4.5,17.5) node {$-2$}; 
\draw [red] (4.5,16.5) node {$3$}; 
\draw (4.5,15.5) node {$8$}; 
\draw (4.5,14.5) node {$13$}; 
\draw (4.5,13.5) node {$18$};

\draw (5.5,18.5) node {\tiny $-10$}; 
\draw (5.5,17.5) node {$-5$}; 
\draw [red] (5.5,16.5) node {$0$}; 
\draw (5.5,15.5) node {$5$}; 
\draw (5.5,14.5) node {$10$}; 
\draw (5.5,13.5) node {$15$}; 

\draw (6.5,18.5) node {\tiny $-13$}; 
\draw (6.5,17.5) node {$-8$}; 
\draw [red] (6.5,16.5) node {$-3$}; 
\draw [red] (6.5,15.5) node {$2$}; 
\draw [red] (6.5,14.5) node {$7$}; 
\draw (6.5,13.5) node {$12$}; 
\draw (6.5,12.5) node {$17$}; 

\draw (7.5,17.5) node {\tiny $-11$}; 
\draw (7.5,16.5) node {$-6$}; 
\draw (7.5,15.5) node {$-1$}; 
\draw [red] (7.5,14.5) node {$4$}; 
\draw (7.5,13.5) node {$9$}; 
\draw (7.5,12.5) node {$14$}; 
\draw (7.5,11.5) node {$19$}; 

\draw (8.5,16.5) node {$-9$}; 
\draw (8.5,15.5) node {$-4$}; 
\draw [red] (8.5,14.5) node {$1$}; 
\draw [red] (8.5,13.5) node {$6$}; 
\draw (8.5,12.5) node {$11$}; 
\draw (8.5,11.5) node {$16$}; 

\draw (9.5,16.5) node {\tiny $-12$}; 
\draw (9.5,15.5) node {$-7$}; 
\draw (9.5,14.5) node {$-2$}; 
\draw [red] (9.5,13.5) node {$3$}; 
\draw (9.5,12.5) node {$8$}; 
\draw (9.5,11.5) node {$13$}; 
\draw (9.5,10.5) node {$18$};

\draw (10.5,15.5) node {\tiny $-10$}; 
\draw (10.5,14.5) node {$-5$}; 
\draw [red] (10.5,13.5) node {$0$}; 
\draw (10.5,12.5) node {$5$}; 
\draw (10.5,11.5) node {$10$}; 
\draw (10.5,10.5) node {$15$}; 

\draw (11.5,15.5) node {\tiny $-13$}; 
\draw (11.5,14.5) node {$-8$}; 
\draw [red] (11.5,13.5) node {$-3$}; 
\draw [red] (11.5,12.5) node {$2$}; 
\draw [red] (11.5,11.5) node {$7$}; 
\draw (11.5,10.5) node {$12$}; 
\draw (11.5,9.5) node {$17$}; 

\draw (12.5,14.5) node {\tiny $-11$}; 
\draw (12.5,13.5) node {$-6$}; 
\draw (12.5,12.5) node {$-1$}; 
\draw [red] (12.5,11.5) node {$4$}; 
\draw (12.5,10.5) node {$9$}; 
\draw (12.5,9.5) node {$14$}; 
\draw (12.5,8.5) node {$19$}; 

\draw (13.5,13.5) node {$-9$}; 
\draw (13.5,12.5) node {$-4$}; 
\draw [red] (13.5,11.5) node {$1$}; 
\draw [red] (13.5,10.5) node {$6$}; 
\draw (13.5,9.5) node {$11$}; 
\draw (13.5,8.5) node {$16$};

\draw (14.5,13.5) node {\tiny $-12$}; 
\draw (14.5,12.5) node {$-7$}; 
\draw (14.5,11.5) node {$-2$}; 
\draw [red] (14.5,10.5) node {$3$}; 
\draw (14.5,9.5) node {$8$}; 
\draw (14.5,8.5) node {$13$}; 
\draw (14.5,7.5) node {$18$};

\draw (15.5,12.5) node {\tiny $-10$}; 
\draw (15.5,11.5) node {$-5$}; 
\draw [red] (15.5,10.5) node {$0$}; 
\draw (15.5,9.5) node {$5$}; 
\draw (15.5,8.5) node {$10$}; 
\draw (15.5,7.5) node {$15$}; 

\draw (16.5,12.5) node {\tiny $-13$}; 
\draw (16.5,11.5) node {$-8$}; 
\draw [red] (16.5,10.5) node {$-3$}; 
\draw [red] (16.5,9.5) node {$2$}; 
\draw [red] (16.5,8.5) node {$7$}; 
\draw (16.5,7.5) node {$12$}; 

\draw (17.5,11.5) node {\tiny $-11$}; 
\draw (17.5,10.5) node {$-6$}; 
\draw (17.5,9.5) node {$-1$}; 
\draw [red] (17.5,8.5) node {$4$}; 
\draw (17.5,7.5) node {$9$}; 

\draw (18.5,10.5) node {$-9$}; 
\draw (18.5,9.5) node {$-4$}; 
\draw [red] (18.5,8.5) node {$1$}; 
\draw [red] (18.5,7.5) node {$6$};

\draw (19.5,10.5) node {\tiny $-12$}; 
\draw (19.5,9.5) node {$-7$}; 
\draw (19.5,8.5) node {$-2$}; 
\draw [red] (19.5,7.5) node {$3$}; 

\draw [very thick, blue] (0,20)--(1,20)--(1,18)--(3,18)--(3,17)--(5,17)--(6,17)--(6,15)--(8,15)--(8,14)--(10,14)--(11,14)--(11,12)--(13,12)--(13,11)--(15,11)--(16,11)--(16,9)--(18,9)--(18,8)--(20,8);
\end{tikzpicture}
\caption{Periodic lattice path corresponding to the $(3,5)$-invariant subset $\Theta=\{0,3,4,5,\ldots\}.$ The $5$-generators of $\Theta$ are $\{0,3,4,6,7\}$, and the $3$-cogenerators are $\{-3,1,2\}$ (both in red).}\label{Figure:periodic path}
\end{figure}

 
\subsection{Equivalence classes of $dn,dm$-invariant subsets}

Let us remind the definitions of the equivalence classes of invariant subsets from \cite{GMVDyck}.   Let $\Theta=\left\{\Theta_0^0,\ldots,\Theta_{l}^0\right\}$ be a collection of $0$-normalized $(n,m)$-invariant subsets. For every $(x_1,\ldots,x_l)\in\mathbb{R}_{\ge 0}^l$ consider
 
\begin{equation*}
\Delta=\Delta(x_1,\ldots,x_l):=\bigcup_{k=0}^l (d\Theta_k^0 + x_k),
\end{equation*}
where $x_0=0.$ If all shift parameters $x_0,\ldots,x_l$ are integers with different remainders modulo $d,$ then $\Delta$ is a $0$-normalized $(nd,md)$-invariant subset. Furthermore, if in addition $l=d-1$ then $\Delta$ is cofinite.

\begin{remark}
In \cite{GMVDyck} only cofinite invariant subsets were considered (i.e. $l=d-1$). However, many results generalize to non-cofinite invariant subsets without any change. 
\end{remark}

For each $\Theta_k^0$ let $S_k^0$ be its skeleton. Consider the space $\mathbb{R}_{\ge 0}^l$ of all possible shifts $x_1,\ldots,x_l.$ Consider the subset $\Sigma_{\Theta}\subset\mathbb{R}_{\ge 0}^l$ consisting of all shifts $x_1,\ldots,x_l$ for which there exists $i$ and $j$ such that 
\begin{equation*}
dS_i^0+x_i\cap dS_j^0+x_j\neq\emptyset.
\end{equation*}

Clearly, $\Sigma_{\Theta}$ is a hyperplane arrangement. We say that two $(nd,md)$-invariant subsets are equivalent if they can be obtained as $\Delta(x_1,\ldots,x_l)$ and $\Delta(y_1,\ldots,y_l)$ for the same collection $\Theta$ and the shifts $(x_1,\ldots,x_l)$ and $(y_1,\ldots,y_l)$ belong to the same connected component of the complement to $\Sigma_\Theta.$ One can show (in fact, it will follow from the construction below) that every connected component of the complement to $\Sigma_{\Theta}$ contains at least one point corresponding to an $(dn,dm)$-invariant subset. We call the connected components of the complement to a hyperplane arrangement the \textit{regions} of the arrangement. See Figure \ref{figure: hyperplane arrangement} for an example of a hyperplane arrangement $\Sigma_\Theta.$

To summarize, in order to get all the equivalence classes of $(dn,dm)$-invariant subsets that have non-trivial intersections with exactly $(l+1)$ congruence classes modulo $d$, one should consider all possible $l+1$-tuples of $0$-normalized $(n,m)$-invariant subsets $\Theta^0_0,\ldots,\Theta^0_l,$ and for each such $(l+1)$-tuple consider the set of regions of $\Sigma_{\Theta}$ in the space of shifts. Furthermore, one should consider these regions up to symmetry: if two of the subsets are equal $\Theta^0_i=\Theta^0_j$ then switching the corresponding shift coordinates $x_i$ and $x_j$ interchanges the connected components corresponding to the same equivalence class of $(dn,dm)$-invariant subsets.

It can be observed (see \cite{GMVDyck}, Lemma $3.16$) that for any region $R$ of $\Sigma_\Theta$ its closure $\overline{R}$ contains the minimal point $(m_1,\ldots, m_l)\in \overline{R},$ such that $m_i=\min\{x_i|(x_1,\ldots,x_l)\in\overline{R}\}$ for all $i.$ Furthermore, it is easy to see that $m_1,\ldots,m_l\in d\mathbb{Z}_{\ge 0}.$ For every $k\in\{0,\ldots, l\}$ set $\Theta_k:=\Theta_k^0+\frac{m_k}{d}$ and $S_k:=S_k^0+\frac{m_k}{d}$ so that $S_k$ is the skeleton of $\Theta_k$. Here $m_0=0.$

\begin{lemma}\label{lemma: intersect vs contain}
One has $S_i\cap S_j=\emptyset$ if and only if either $\Theta_i\subset \Theta_j+n+m$ or $\Theta_j\subset \Theta_i+n+m.$
\end{lemma}

\begin{proof} According to Lemma \ref{lemma: skeleton vs periodic path}, $S_i\cap S_j=\emptyset$ is equivalent to the fact that the corresponding periodic lattice paths don't intersect, which implies that one of them sits weakly below the other one shifted one south and one to the west ($+n+m$), which is equivalent to $\Theta_i\subset \Theta_j+n+m$ or $\Theta_j\subset \Theta_i+n+m$.
\end{proof}

Note that the hyperplanes of $\Sigma_\Theta$ that pass through a point $(m_1,\ldots,m_l)$ are of the form $x_i-x_j=m_i-m_j$ with one such hyperplane for each pair $i>j$ such that the corresponding skeletons intersect: $S_i\cap S_j\neq\emptyset$ (here $j$ might be zero, in which case $m_0=x_0=0$). The point $(m_1,\ldots,m_l)$ might be the minimal point of closures of multiple regions of $\Sigma_\Theta.$ In order to pick one such region $R,$ for each hyperplane $x_i-x_j=m_i-m_j$ of $\Sigma_\Theta$ one should pick on which side of that hyperplane the region lies. This information can be encoded into a digraph $G$ with vertices $\Theta_0,\ldots,\Theta_l$ and edges connecting subsets $\Theta_i,\Theta_j$ with intersecting skeletons, going in the direction of $\Theta_i$ whenever $R\subset \{x_i-m_i>x_j-m_j\}$. Note that $G$ is acyclic and defines a partial order on $\Theta$. The minimality of the point $(m_1,\ldots,m_l)$ implies also that $\Theta_0$ is the unique source.

\begin{figure}
\begin{tikzpicture}[scale=0.3]

\filldraw[gray!50!white] (0,18)--(3,18)--(3,33)--(0,33)--(0,18);
\filldraw[gray!50!white] (0,0)--(3,0)--(3,3)--(0,3)--(0,0);

\draw (-1,33) node {$x_2$};
\draw (39,-1) node {$x_1$};

\foreach \i in {0,...,12}
	\foreach \j in {0,...,10}{
		\filldraw (3*\i+1,3*\j+2) circle(1pt); 
		\filldraw (3*\i+2,3*\j+1) circle(1pt); 
		}
\foreach \i in {0,...,6}
	\foreach \j in {0,...,4}{
		\filldraw[red] (3*\i+1,3*\j+2) circle(4pt); 
		\filldraw[red] (3*\i+2,3*\j+1) circle(4pt);
		\filldraw[red] (3*\i+3*\j+2,3*\j+1) circle(4pt);
		}
\foreach \i in {1,...,6}{
	\filldraw[red] (3*\i+1,17) circle(4pt);
	}
\foreach \i in {2,...,12}{
	\filldraw[red] (3*\i+2,19) circle(4pt);
	}
\foreach \i in {2,...,6}{
	\filldraw[red] (3*\i+1,20) circle(4pt);
	}
\foreach \i in {3,...,6}{
	\filldraw[red] (3*\i+1,23) circle(4pt);
	}
\foreach \i in {4,...,6}{
	\filldraw[red] (3*\i+1,26) circle(4pt);
	}
\foreach \i in {5,...,6}{
	\filldraw[red] (3*\i+1,29) circle(4pt);
	}
\filldraw[red] (19,32) circle(4pt);
\filldraw[red] (2,19) circle(4pt);
\filldraw[red] (5,19) circle(4pt);
	
\filldraw[blue] (1,20) circle(4pt);
\filldraw[blue] (4,20) circle(4pt);

\draw [->] (0,0)--(40,0);
\draw (0,3)--(39,3);
\draw (0,6)--(39,6);
\draw (0,9)--(39,9);
\draw (0,12)--(39,12);
\draw (0,18)--(39,18);

\draw [->] (0,0)--(0,34);
\draw (3,0)--(3,33);
\draw (6,0)--(6,33);
\draw (9,0)--(9,33);
\draw (12,0)--(12,33);
\draw (15,0)--(15,33);
\draw (18,0)--(18,33);

\draw (0,12)--(21,33);
\draw (0,9)--(24,33);
\draw (0,6)--(27,33);
\draw (0,3)--(30,33);
\draw (0,0)--(33,33);
\draw (3,0)--(36,33);
\draw (6,0)--(39,33);
\draw (9,0)--(39,30);
\draw (12,0)--(39,27);
\draw (15,0)--(39,24);
\draw (18,0)--(39,21);

\filldraw (0,0) circle(6pt);
\filldraw (0,18) circle(6pt);

\draw (-2,0) node {$(0,0)$};
\draw (-2,18) node {$(0,18)$};

\end{tikzpicture}
\caption{The hyperplane arrangement $\Sigma_\Theta$ for the collection of $(3,2)$-invariant subsets from Examples \ref{example: min reps d=3}, \ref{example: shift (0,18)}, and \ref{ex: admiss reps d=3}. The lattice of dots are the integer shifts with different remainders modulo $3.$ They correspond to the $(9,6)$-invariant subsets. The red dots correspond to the admissible subsets, and the blue dots correspond to the minimal representatives that are not admissible. The three shaded regions are the regions considered in Examples, and the bold black dots are the minimal shifts on the boundaries of those regions.}\label{figure: hyperplane arrangement}
\end{figure}
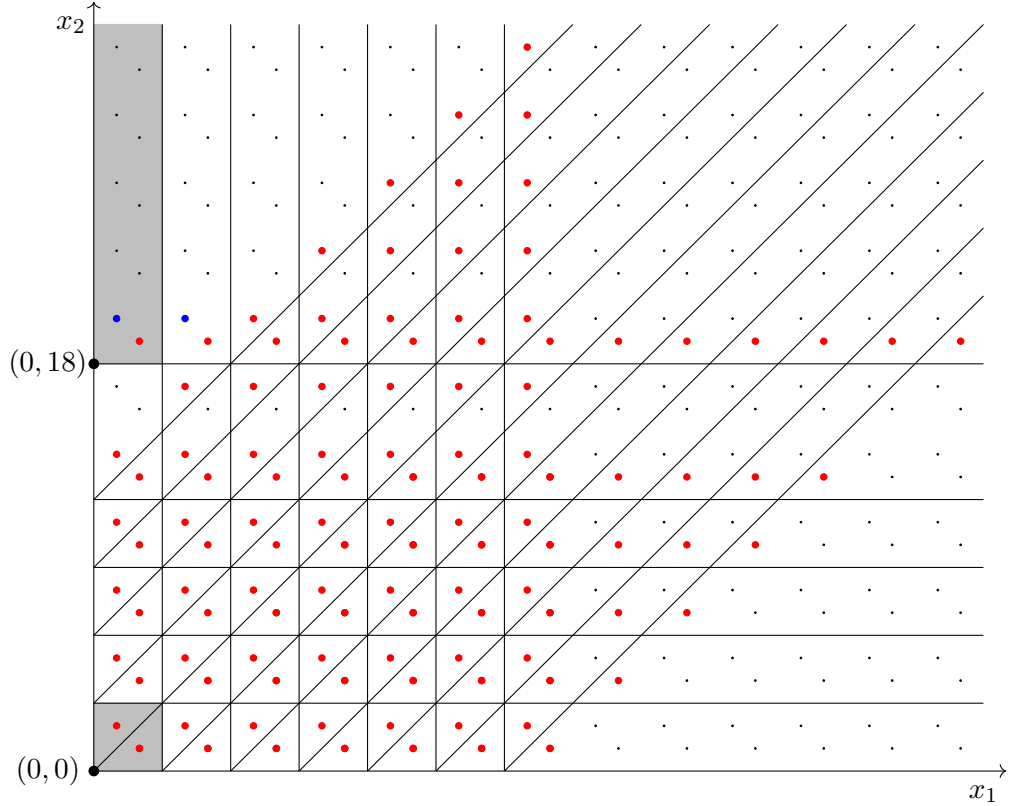

\begin{definition}
We call the graph $G$ described above the \textit{intersection graph} of the skeletons.
\end{definition}

Let $\sigma=(\sigma_0,\ldots,\sigma_l)\in S_{l+1}$ be a permutation, such that $\Theta_{\sigma_0},\ldots,\Theta_{\sigma_l}$ are weakly increasing in the partial order defined by the acyclic orientation on $G$ (in particular, $\sigma_0=0$). Then 

\begin{equation*}
\Delta(m_0+\sigma_0,m_1+\sigma_1,\ldots,m_l+\sigma_l)=\bigcup d\Theta_k+\sigma_k
\end{equation*}
is a $(dn,dm)$-invariant subset, such that the corresponding shift belongs to the region $R.$ In particular, it follows that $R$ corresponds to a non-empty equivalence class. The subsets constructed as above are called the {\em minimal representatives} of the class corresponding to $R.$ Note that there might be more than one minimal representative: they are parametrized by the linearizations of the partial order on $\Theta$ given by the acyclic orientation on $G$.

It follows also that if $(l_1,\ldots,l_l)\in d\mathbb{Z}_{\ge 0}^l\subset\mathbb{R}_{\ge 0}^l$ is a shift, such that the corresponding intersection graph of skeletons is connected, then this shift is the minimal point of at least one of the regions. The regions with that minimal point are then parametrized by the acyclic orientations of the graph, such that $\Theta_0$ is the unique source.

\begin{example}\label{example: min reps d=3}
Let $(n,m)=(3,2)$ and $d=3.$ Let also
\begin{align*}
\Theta^0_0&=\Theta^0_2=\{0,2,3,\ldots\}=\mathbb{Z}_{\ge 0}\setminus \{1\},\\
\Theta^0_1&=\{0,1,2,\ldots\}=\mathbb{Z}_{\ge 0},
\end{align*}
so that 
\begin{align*}
S^0_0&=S^0_2=\{-2,0,1,2,4\},\\
S^0_1&=\{-2,-1,0,1,2\}.
\end{align*}
In Figure \ref{figure: hyperplane arrangement} we present the corresponding hyperplane arrangement $\Sigma_\Theta.$

For the zero shift $(m_1,m_2)=(0,0)$ all three skeletons pairwise intersect. Therefore, the intersection graph is complete and, in particular, connected. There are two ways to orient the graph so that $S_0^0$ is the unique source: the edges connecting $S_0^0$ to $S_1^0$ and $S_2^0$ have to be oriented pointing away from $S_0^0,$ but the edge connecting $S_1^0$ and $S_2^0$ can be oriented either way. Therefore, this is the minimal point on the boundary of two regions. Both orientations determine linear orders, therefore the corresponding regions contain unique minimal representatives:
\begin{equation*}
3\Theta_0^0\sqcup(3\Theta_1^0+2)\sqcup(3\Theta_2^0+1)=\{0,1,2,5,6,\ldots\}=\mathbb{Z}_{\ge 0}\setminus\{3,4\},
\end{equation*}
for $S_2^0\to S_1^0,$ and
\begin{equation*}
3\Theta_0^0\sqcup(3\Theta_1^0+1)\sqcup(3\Theta_2^0+2)=\{0,1,2,4,6,\ldots\}=\mathbb{Z}_{\ge 0}\setminus\{3,5\},
\end{equation*}
for $S_1^0\to S_2^0.$ (See the two shaded  triangular regions adjacent to the origin in Figure \ref{figure: hyperplane arrangement}.)
\end{example}

\begin{example}\label{example: shift (0,18)}
Consider now the shift $(m_1,m_2)=(0,3\times 6)=(0,18).$ One gets
\begin{align*}
S_0&=S^0_0=\{-2,0,1,2,4\},\\
S_1&=S^0_1=\{-2,-1,0,1,2\},\\
S_2&=S^0_2+6=\{4,6,7,8,10\}.
\end{align*}
Note that $S_1\cap S_2=\emptyset$ while both $S_1$ and $S_2$ have non-trivial intersections with $S_0.$ The corresponding graph has a unique orientation, such that $S_0$ is the unique source. (See Figure \ref{Figure: graph d=3} on the left for the graph, and the shaded region adjacent to the point $(0,18)$ in Figure \ref{figure: hyperplane arrangement} for the corresponding region.)

\begin{figure}
\begin{center}

\begin{tikzpicture}[
        > = stealth, 
            shorten > = 2pt, 
            shorten < = 2pt, 
            auto,
            semithick 
        ]
\tikzset{myptr/.style={decoration={markings,mark=at position .75 with
    {\arrow[scale=2,>=stealth]{>}}},postaction={decorate}}}
\tikzset{state/.style={circle,draw,minimum size=2mm, inner sep=0pt}, }

    \node[state, label=right:{ \{-2,0,1,2,4\}}] (0) {};
    \node[state, label=left: { 
\{4,6,7,8,10\}}] (1) [above left of=0, xshift=-4mm, yshift=15mm]  {};
    \node[state, label=right:{
\{-2,-1,0,1,2\}}] (2) [above right of=0, xshift=4mm, yshift=15mm] {};

    \draw[myptr] (0) -- (2);
    \draw[myptr] (0) -- (1);

\end{tikzpicture}
\begin{tikzpicture}[
        > = stealth, 
            shorten > = 2pt, 
            shorten < = 2pt, 
            auto,
            semithick 
        ]
\tikzset{myptr/.style={decoration={markings,mark=at position .75 with
    {\arrow[scale=2,>=stealth]{>}}},postaction={decorate}}}
\tikzset{state/.style={circle,draw,minimum size=2mm, inner sep=0pt}, }

    \node[state, label=right:{ $\mathbb{Z}_{\ge 0}\setminus \{1\}$}] (0) {};
    \node[state, label=left: { 
$\mathbb{Z}_{\ge 6}\setminus \{7\}$}] (1) [above left of=0, xshift=-4mm, yshift=15mm]  {};
    \node[state, label=right:{
$\mathbb{Z}_{\ge 0}$}] (2) [above right of=0, xshift=4mm, yshift=15mm] {};

    \draw[color=blue, myptr] (0) -- (2);
    \draw[color=blue, myptr] (0) -- (1);
    \draw[color=green, myptr] (1) -- (2);

\end{tikzpicture}
\end{center}
\caption{On the left: the digraph $G$ corresponding to the minimal point on the boundary of a region with vertices labeled by the skeletons. On the right: the corresponding bicolored digraph with vertices labeled by the $(2,3)$-invariant subsets.} 
\label{Figure: graph d=3}

\end{figure}
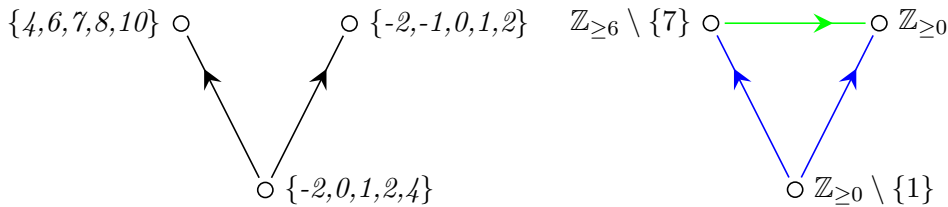
The permutation $\sigma$ in this case can be either $(0,1,2)$ or $(0,2,1),$ giving rise to two minimal representatives:

\begin{equation*}
3\Theta_0^0\sqcup (3\Theta_1^0+1)\sqcup (3\Theta_2^0+18+2)=\mathbb{Z}_{\ge 0}\setminus\{2,3,5,8,11,14,17,23\},
\end{equation*}
for $\sigma=(0,1,2),$ and
\begin{equation*}
3\Theta_0^0\sqcup (3\Theta_1^0+2)\sqcup (3\Theta_2^0+18+1)=\mathbb{Z}_{\ge 0}\setminus\{1,3,4,7,10,13,16,22\},
\end{equation*}
for $\sigma=(0,2,1).$ (See the blue and the red points in the shaded rectangular region in Figure \ref{figure: hyperplane arrangement}.)
\end{example}

\subsection{Admissible representatives}

Let $\Delta$ be a $0$-normalized $(dn,dm)$-invariant subset such that
\begin{align*}
&\Delta\cap (d\mathbb{Z}+r)\neq \emptyset, r\in\{0,1,\ldots,l\},\\
&\Delta\cap (d\mathbb{Z}+r)= \emptyset, r\in\{l+1,l+1,\ldots,d-1\}.
\end{align*}
In the spirit of the above notations, let 

\begin{equation*}
\Delta=\bigcup_{k} d\Theta_k^0 + x_k + k,
\end{equation*} 
where $(x_1,\ldots,x_l)\in d\mathbb{Z}_{\ge 0}^l\subset\mathbb{R}_{\ge 0}^l,$ $x_0=0,$ and $\Theta_0^0,\ldots,\Theta_l^0$ are $0$-normalized $(n,m)$-invariant subsets. Denote $\Theta_k:=\Theta_k^0+\frac{x_k}{d}$ for all $k\in\{0,\ldots, l\}.$  

One can generalize the notion of an admissible subset (see Definition \ref{def: admissible} and Lemma \ref{lem:susp}) to the case of non-cofinite $(nd,md)$-invariant subsets of the type described above in a straightforward way. We get that $\Delta$ is admissible if for every $k\in\{0,\ldots,l-1\}$ one has
\begin{align*}
&\Delta\cap (d\mathbb{Z}+k+1)\nsubseteq \Delta\cap (d\mathbb{Z}+k)+dn+dm+1,\\
&\Leftrightarrow\quad d\Theta_{k+1}+k+1\nsubseteq d\Theta_k+k+dn+dm+1,\\
&\Leftrightarrow\quad \Theta_{k+1}\nsubseteq \Theta_k+n+m.
\end{align*}

In the view of Lemma \ref{lemma: intersect vs contain}, this is also equivalent to saying that $\Delta$ is admissible if and only if for every $k\in\{0,\ldots,l-1\}$ either $S_k\cap S_{k+1}\neq\emptyset$ or $\Theta_k\subset \Theta_{k+1}+n+m,$ where $S_0,\ldots,S_l$ are the skeletons of $\Theta_0,\ldots,\Theta_l.$

\begin{lemma}\label{lemma:admissible vs minimal}
If $\Delta$ is admissible then it is a minimal representative of an equivalence class.
\end{lemma}

\begin{proof}
According to the constructions above, all we need to check is that $(x_1,\ldots,x_l)$ is the minimal point on the boundary of a region, which is equivalent to showing that the intersection graph of the skeletons of $\Theta_0,\ldots,\Theta_l$ is connected. 

Indeed, suppose it is not connected. Let $k$ be the minimal number such that $\Theta_k$ is not in the same connected component as $\Theta_0.$ Note that $\Theta_0\nsubseteq \Theta_k+n+m$ because $0\in\Theta_0$ and $\min(\Theta_k+n+m)\ge n+m>0.$ Furthermore, $S_0\cap S_k=\emptyset,$ because $\Theta_k$ is not connected to $\Theta_0$ in the intersection graph. Therefore, according to Lemma \ref{lemma: intersect vs contain}, $\Theta_k\subset \Theta_0+n+m.$ 

Suppose now that $\Theta_i$ and $\Theta_j$ are connected by an edge in the intersection graph, while $\Theta_k$ is not connected to either of them. That is equivalent to saying that the periodic lattice paths corresponding to $\Theta_i$ and $\Theta_j$ intersect, while the periodic path corresponding to $\Theta_k$ does not intersect either of them. It follows then that if the periodic path corresponding to $\Theta_k$ is either strictly below both periodic paths corresponding to $\Theta_i$ and $\Theta_j,$ or it is strictly above both of them. In other words, if $\Theta_k\subset \Theta_i +n+m$ then also $\Theta_k\subset \Theta_j +n+m.$

Since $\Theta_0$ and $\Theta_{k-1}$ are in the same connected component of the intersection graph and $\Theta_k\subset \Theta_0+n+m,$ applying the argument from the previous paragraph several times, we obtain that $\Theta_k\subset\Theta_{k-1}+n+m,$ which means that $\Delta$ is not admissible. Contradiction.

\end{proof}

\begin{theorem}
\label{thm: unique admissible}
Every equivalence class contains a unique admissible representative.
\end{theorem}

\begin{proof}
According to Lemma \ref{lemma:admissible vs minimal} it is enough to show that for every equivalence class exactly one of the minimal representatives is admissible. We now follow the notations from the previous section. We get the intersection digraph of the skeletons of $\Theta_0,\ldots,\Theta_{d-1},$ which is acyclic and contains the unique source $\Theta_0.$ Consider a minimal representative

\begin{equation*}
\Delta(m_0+\sigma_0,m_1+\sigma_1,\ldots,m_l+\sigma_l)=\bigcup d\Theta_k+\sigma_k.
\end{equation*}
(Here $\sigma\in S_{l+1}$ is such that $\Theta_{\sigma_0},\ldots,\Theta_{\sigma_l}$ is weakly increasing in the order defined by the orientation.) This minimal representative is admissible if and only if whenever $S_{\sigma_k}\cap S_{\sigma_{k+1}}=\emptyset,$ one has $\Theta_{\sigma_k}\subset\Theta_{\sigma_{k+1}}+n+m.$ 

Let us enhance the digraph $G$ as follows. Let's color the edges of the original digraph in blue, and whenever two vertices $\Theta_i$ and $\Theta_j$ are not connected by a blue edge, connect them by a green edge oriented in such a way that if $\Theta_i\subset \Theta_j+n+m$ then $\Theta_i\rightarrow \Theta_j.$ We obtain a complete graph with edges colored in two colors: blue and green, and oriented in such a way that
\begin{enumerate}
\item Each color is acyclic.
\item The green color is transitive, i.e. if one has green edges $\Theta_i\rightarrow \Theta_j$ and $\Theta_j\rightarrow \Theta_k,$ then one also has a green edge $\Theta_i\rightarrow \Theta_k.$ 
\end{enumerate}

The proof of the Theorem is reduced to the following claim: such a graph contains a unique oriented path passing through very vertex exactly once (possibly using both colors) and such that it is weakly monotone with respect to the partial order defined by the blue color.

Existence. We construct the path step by step. Suppose that $\Theta_{\sigma_0}\to\Theta_{\sigma_1}\to\ldots\to\Theta_{\sigma_k}$ is an initial segment of the path already constructed, and suppose that it satisfies the required conditions, i.e. all the consecutive edges are correctly oriented, the path is weakly monotone with respect to the blue partial order, and there are no blue edges going from one of the unused vertices to one of the used vertices (otherwise it would not be possible to extend the path to that vertex because of the monotonicity condition). Consider all sources of the restriction of the blue graph to the unused vertices. Note that it is a non-empty set unless $k=l,$ in which case we are done. These vertices can only be connected by green edges. In particular, they are totally ordered with respect to the green orientation. Let $\sigma_{k+1}$ be such that $\Theta_{\sigma_{k+1}}$ the smallest of these vertices with respect to the green order. The blue monotonicity condition for this extension of the path is satisfied by construction. All we need to check is that 
there cannot be a green arrow $\Theta_{\sigma_{k+1}}\to\Theta_{\sigma_k}.$ Indeed, if that would be the case then $\Theta_{\sigma_{k+1}}$ would have been added before $\Theta_{\sigma_k}$ according to our algorithm. We repeat this step until we obtain a path going through every vertex.

Uniqueness. Suppose that we have a path $\Theta_{\omega_0}\to\Theta_{\omega_1}\to\ldots\to\Theta_{\omega_l}$ satisfying the required conditions and different from the one constructed using the algorithm above. That means that there exist two indexes $1\le i<j\le l$ such that $\Theta_{\omega_j}$ should have been added instead of $\Theta_{\omega_i}$ according to the algorithm. In particular, it implies that $\Theta_{\omega_j}$ cannot be connected to any of the vertices $\Theta_{\omega_i},\Theta_{\omega_{i+1}},\ldots,\Theta_{\omega_{j-1}}$ by blue edges in either direction, and we have a green edge $\Theta_{\sigma_j}\to\Theta_{\sigma_i}.$ We claim that all edges connecting $\Theta_{\sigma_j}$ with the vertices $\Theta_{\omega_i},\Theta_{\omega_{i+1}},\ldots,\Theta_{\omega_{j-1}}$ are oriented FROM $\Theta_{\sigma_j}.$ Indeed, otherwise let $\Theta_{\sigma_k},\ i<k<j$ be the first vertex such that we have a green edge $\Theta_{\sigma_k}\to\Theta_{\sigma_j}.$ Then  we also have a green edge $\Theta_{\sigma_j}\to\Theta_{\sigma_{k-1}},$ and by transitivity of the green edges we get a green edge $\Theta_{\sigma_k}\to\Theta_{\sigma_{k-1}},$ which contradicts the conditions on the path.
\end{proof}

\begin{example}\label{ex: admiss reps d=3}
Continuing Example \ref{example: shift (0,18)}, note that 
\begin{align*}
\Theta_1&=\Theta^0_1=\mathbb{Z}_{\ge 0},\\
\Theta_2&=\Theta^0_2+6=\{6,8,9,\ldots\}=\mathbb{Z}_{\ge 6}\setminus \{7\}.
\end{align*}

Therefore, one gets $\Theta_2\subset \Theta_1+n+m=\Theta_1+5,$ so the green edge points from $\Theta_2$ to $\Theta_1$ (See Figure \ref{Figure: graph d=3} on the right). One concludes that the minimal representative corresponding to $\sigma=(0,2,1)$ is admissible, while the minimal representative corresponding to $\sigma=(0,1,2)$ is not (see the blue versus red dots in the shaded rectangular region in Figure \ref{figure: hyperplane arrangement}).
\end{example}

\begin{example} In Figure \ref{Figure: graph d=4} on the left we see an example of a digraph corresponding to the minimal point on the boundary of a region for $(n,m)=(2,3)$ and $d=4$ from \cite{GMVDyck}. The vertices are labeled by the skeletons of the corresponding $(2,3)$-invariant subsets. Note that there are three distinct linearizations of the partial order defined by the orientation of the graph, each corresponding to a minimal representative of the corresponding equivalence class.

On the right in Figure \ref{Figure: graph d=4} we see the corresponding bicolored digraph with vertices labeled by the $(2,3)$-invariant subsets. Indeed, one has $\mathbb{Z}_{\ge 6}\subset \mathbb{Z}_{\ge 0}+5$ and $\mathbb{Z}_{\ge 6}\setminus \{7\}\subset \mathbb{Z}_{\ge 0}+5,$ which determine the orientations of the green edges. It follows that only one of the three linearizations corresponds to an admissible subset:

\begin{align*}
\Delta=&\left[4(\mathbb{Z}_{\ge 0}\setminus \{1\})\right]\sqcup \left[4(\mathbb{Z}_{\ge 6}\setminus \{7\})+1\right]\sqcup \left[4\mathbb{Z}_{\ge 6}+2\right]\sqcup \left[4\mathbb{Z}_{\ge 0}+3\right]\\
&=\mathbb{Z}_{\ge 0}\setminus \{1,2,4,5,6,9,10,13,14,17,18,21,22,29\}.
\end{align*}

\begin{figure}
\begin{center}

\begin{tikzpicture}[
        > = stealth, 
            shorten > = 2pt, 
            shorten < = 2pt, 
            auto,
            semithick 
        ]
\tikzset{myptr/.style={decoration={markings,mark=at position .75 with
    {\arrow[scale=2,>=stealth]{>}}},postaction={decorate}}}
\tikzset{state/.style={circle,draw,minimum size=2mm, inner sep=0pt}, }

    \node[state, label=right:{ \{4,5,6,7,8\}}] (3) {};
    \node[state, label=right: { 
\{4,6,7,8,10\}}] (2) [below right  of=3, xshift=4mm, yshift=-15mm]  {};
    \node[state, label=left:{ 
\{-2,-1,0,1,2\}}] (1) [below left  of=3, xshift=-4mm, yshift=-15mm] {};
    \node[state, label=right:{
\{-2,0,1,2,4\}}] (0) [below of=3, yshift=-30mm] {};

    \draw[myptr] (2) -- (3);
    \draw[myptr] (0) -- (3);
    \draw[myptr] (0) -- (2);
    \draw[myptr] (0) -- (1);

\end{tikzpicture}
\begin{tikzpicture}[
        > = stealth, 
            shorten > = 2pt, 
            shorten < = 2pt, 
            auto,
            semithick 
        ]
\tikzset{myptr/.style={decoration={markings,mark=at position .75 with
    {\arrow[scale=2,>=stealth]{>}}},postaction={decorate}}}
\tikzset{state/.style={circle,draw,minimum size=2mm, inner sep=0pt}, }

    \node[state, label=right:{ $\mathbb{Z}_{\ge 6}$}] (3) {};
    \node[state, label=right: { 
$\mathbb{Z}_{\ge 6}\setminus \{7\}$}] (2) [below right  of=3, xshift=4mm, yshift=-15mm]  {};
    \node[state, label=left:{ 
$\mathbb{Z}_{\ge 0}$}] (1) [below left  of=3, xshift=-4mm, yshift=-15mm] {};
    \node[state, label=right:{
$\mathbb{Z}_{\ge 0}\setminus \{1\}$}] (0) [below of=3, yshift=-30mm] {};

    \draw[color=blue, myptr] (2) -- (3);
    \draw[color=blue, myptr] (0) -- (3);
    \draw[color=blue, myptr] (0) -- (2);
    \draw[color=blue, myptr] (0) -- (1);
    \draw[color=green, myptr] (2) -- (1);
    \draw[color=green, myptr] (3) -- (1);

\end{tikzpicture}
\end{center}
\caption{On the left: the digraph with vertices labeled by the skeletons of the $(3,2)$-invariant subsets. On the right: the corresponding bicolored graph with vertices labeled by the $(2,3)$-invariant subsets.
\label{Figure: graph d=4}
}
\end{figure}
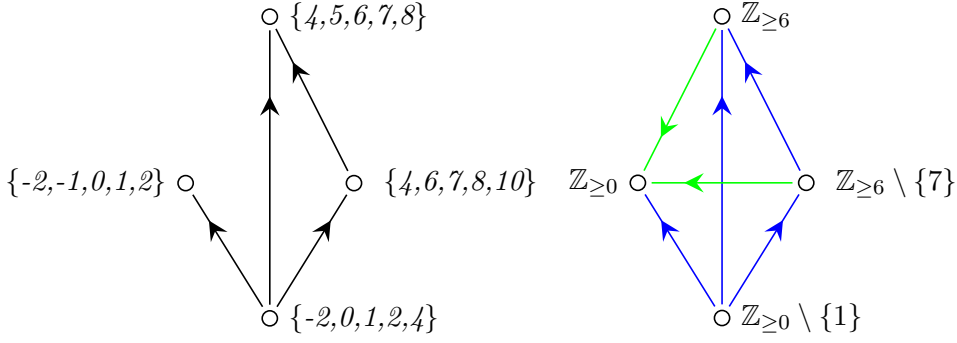

\end{example}

Let $\Inv(nd,md)^{l+1}$ be the set of $0$-normalized $(nd,md)$-invariant subsets such that 
\begin{align*}
&\Delta\cap (d\mathbb{Z}+r)\neq \emptyset, r\in\{0,1,\ldots,l\},\\
&\Delta\cap (d\mathbb{Z}+r)= \emptyset, r\in\{l+1,l+1,\ldots,d-1\},
\end{align*}
as above. The following Theorem is the main result of \cite{GMVDyck}:

\begin{theorem}[\cite{GMVDyck}]
\label{thm: GMVDyck}
There exists a bijection $\mathcal{D}:\Inv(nd,md)^{l+1}/\sim \to \Dyck(n(l+1),m(l+1)),$ where $\sim$ is the equivalence relation defined above, and $\Dyck(n(l+1),m(l+1))$ is the set of $(n(l+1),m(l+1))$-Dyck paths. In addition, in the cofinite case $l=d-1$ one gets
\begin{equation*}
\dim \Delta=\delta - \dinv(\mathcal{D}(\Delta))
\end{equation*}
for all $\Delta\in \Inv(nd,md):=\Inv(nd,md)^d.$
\end{theorem}

\begin{remark}
In \cite{GMVDyck} only the cofinite case $l=d-1$ is considered, but the construction of the bijection $\mathcal{D}$ generalizes to the non-cofinite case $l<d-1$ without any change.
\end{remark}

\subsection{Proof of Theorem \ref{thm: intro main}}

We combine all of the above results to prove Theorem  \ref{thm: intro main}. The compactified Jacobian $\overline{JC}$ is stratified into locally closed subsets  $J_{\Delta}$ and by Proposition \ref{prop:adm-tria} they are isomorphic to affine spaces of dimension $\dim(\Delta)$  if $\Delta$ is admissible and empty otherwise. Therefore the Poincar\'e polynomial has the form
$$
P(T)=\sum_{\Delta\ \mathrm{admissible}}T^{2\dim(\Delta)}.
$$
Next, we consider the infinite set $\Inv(nd,md)$ of all $(nd,md)$-invariant subsets in $\Z_{\ge 0}$ and the equivalence relation on it. By Theorem \ref{thm: unique admissible} in each equivalence class there is a unique admissible representative. Furthermore, by Lemma \ref{lem: dim gen cogen} the ``combinatorial dimension" $\dim(\Delta)$ depends only on the order of generators and cogenerators, and hence is constant on each equivalence class. Therefore we can write
$$
P(T)=\sum_{\Delta\in \Inv(nd,md)/\sim}T^{2\dim(\Delta)}.
$$
Finally, by Theorem \ref{thm: GMVDyck} there is a bijection between the equivalence classes and Dyck paths in $(nd)\times (md)$ rectangle, and the statistic $\dim(\Delta)$ on the former corresponds to the statistic $\codinv=\delta-\dinv$ on the latter, so
$$
P(T)=\sum_{D\in \Dyck(nd,md)}T^{2(\delta-\dinv(d))}.
$$

\section{Rational Shuffle Theorem and generic curves}
\label{sec: daha}

\subsection{Elliptic Hall algebra}

We briefly recall some notations for the elliptic Hall algebra \cite{BS}, and refer the reader to \cite{BGLX,GN,Mellit, Negut} for more precise statements and details.

Let $\Lambda=\Lambda_{Q,T}$ be the ring of symmetric functions in infinitely many variables with coefficients in $\mathbb{Q}(Q,T)$. The elliptic Hall algebra $\cE$ acts on $\Lambda$ and satisfies the following properties:
\begin{enumerate}
\item The multiplication operators by power sum symmetric functions $p_k$ correspond (up to a scalar) to the generators $P_{k,0}$ of $\cE$.
\item By expanding an arbitrary symmetric function $f$ in $p_k$ and replacing them by $P_{k,0}$, one associates to $f$ an operator in $\cE$.
\item The universal cover of the group $\mathrm{SL}(2,\Z)$ acts on $\cE$ by automorphisms. For coprime $m$ and $n$ we denote by $\gamma_{n,m}$ an element of $\mathrm{SL}(2,\Z)$ such that $\gamma_{n,m}(1,0)=(n,m)$, and write $P_{kn,km}=\gamma_{n,m}(P_{k,0})$. Note that the choice of $\gamma_{n,m}$ is not unique, but different choices lead to the same operators, up to a scalar. The algebra $\cE$ is generated by $P_{kn,km}$ for all possible $(kn,km)$.
\item  The algebra $\cE$ is graded, and the grading is compatible with the grading on $\Lambda$. If $A\in \cE$ is a degree $d$ operator then it shifts the degrees on $\Lambda$ by $d$, in particular $A(1)$ is a symmetric function of degree $d$.
The generator $P_{kn,km}$ has degree $kn$.
\item For $n=0$, the generators $P_{0,m}$ have degree 0 and correspond to (generalized) Macdonald operators on $\Lambda$ which 
have modified Macdonald polynomials as eigenvectors. 
\item One has $P_{kn,km+kn}=\nabla P_{kn,km} \nabla^{-1}$, where $\nabla$ is a certain degree-preserving operator on symmetric functions which is diagonal in the modified Macdonald basis.
\end{enumerate}

We also note that $\cE$ and its representation $\Lambda$ can be thought of as limits of spherical double affine Hecke algebras  and their representations in symmetric polynomials \cite{SV}. We refer to \cite{GN} for the precise normalization conventions for this limit, which we implicitly use here. In particular, the Cherednik-Danilenko conjecture is formulated in terms of DAHA, but we reformulate it here in terms of the elliptic Hall algebra.

\subsection{Rational Shuffle Theorem}

We can use the above notations to explore a special case of the Compositional Rational Shuffle Theorem conjectured in \cite{BGLX} and proved in \cite{Mellit}, see also \cite{BHMPSline}.

Let 
$$
S_i=S_i(nd,md)=\left\lceil\frac{im}{n}\right\rceil-\left\lceil\frac{(i-1)m}{n}\right\rceil,\ 1\le i\le nd.
$$
 
\begin{theorem}
\label{thm: daha}
Consider the operator $\gamma_{n,m}(e_d)$ in $\cE$, and the symmetric function $\gamma_{n,m}(e_d)(1)$.
We have
$$
C_{md,nd}(Q,T)=(\gamma_{n,m}(e_d)(1),e_{nd})=
$$
$$
\sum_{\mathbf{T}\in \SYT(nd)}\frac{z_1^{S_1}\cdots z_{nd}^{S_{nd}}}{\prod_{i=2}^{nd}(1-z_i^{-1})(1-QTz_{i-1}/z_i)}\prod_{i<j}\frac{(1-z_i/z_j)(1-QTz_i/z_j)}{(1-Qz_i/z_j)(1-Tz_i/z_j)}
$$
where $C_{md,nd}(Q,T)$ are defined by \eqref{eq: def qtcat} and  $z_i$ is the $(Q,T)$-content of a box labeled by $i$ in the standard Young tableau (SYT) $\mathbf{T}$ of size $nd$.
\end{theorem}


\begin{proof}
The identity $C_{md,nd}(Q,T)=(\gamma_{n,m}(e_d)(1),e_{nd})$ follows from \cite[Conjecture 3.1]{BGLX}\footnote{Note that \cite{BGLX} use slightly different normalization of operators which leads to additional sign.}  which is a special case of Compositional Rational Shuffle Conjecture  \cite[Conjecture 3.3]{BGLX} proved by Mellit in \cite{Mellit}.
The explicit formula for  $\gamma_{n,m}(e_d)(1)$  follows from \cite[Proposition 6.13]{Negut}.
\end{proof}

\subsection{Cherednik-Danilenko conjecture}
\label{sec: CD}

We can now explain the connection between our results and a conjecture of Cherednik and Danilenko \cite{CheD}. The latter is phrased for arbitrary algebraic plane curve singularities  with Puiseux expansion  
$$
y=b_1x^{\frac{m_1}{r_1}}+b_2x^{\frac{m_2}{r_1r_2}}+b_3x^{\frac{m_3}{r_1r_2r_3}}+\ldots,\ \quad b_i\neq 0
$$
Here we assume $\GCD(m_i,r_1\cdots r_i)=1$.
The exponents are related to characteristic pairs $(r_i,s_i)$ by the equations $m_1=s_1,m_i=s_i+r_im_{i-1}\ (i>1)$.

Given a sequence of characteristic pairs $(r_1,s_1),\ldots,(r_{\ell},s_{\ell})$, the authors of \cite{CheD} define a sequence of symmetric functions as follows. Start from degree one polynomial $f_{\ell+1}:=p_1\in \Lambda$, view it as a multiplication operator in $\cE$ and apply the automorphism $\gamma_{r_\ell,s_\ell}$. The resulting operator $\gamma_{r_\ell,s_\ell}(f_{\ell+1})=P_{r_{\ell},s_{\ell}}$ has degree $r_{\ell}$, and  we can consider the symmetric polynomial
$$
f_{\ell}=\gamma_{r_\ell,s_\ell}(f_{\ell+1})(1)=P_{r_{\ell},s_{\ell}}(1)
$$
of degree $r_{\ell}$. Similarly, one can define an operator $\gamma_{r_{\ell-1},s_{\ell-1}}(f_{\ell})$ and a symmetric function 
$$
f_{\ell-1}=\gamma_{r_{\ell-1},s_{\ell-1}}(f_{\ell})(1)
$$
of degree $r_{\ell}r_{\ell-1}$ and so on. 

\begin{conjecture}[\cite{CheD}]
\label{conj: CD}
Let $f_1$ be the symmetric function of degree $r_1\cdots r_{\ell}$ obtained by the above procedure. The specialization of
$(f_1, e_{r_1\cdots r_\ell})$ at $T=1$ is well defined and agrees with the Poincar\'e polynomial of the compactified Jacobian of a plane curve singularity with characteristic pairs $(r_1,s_1),\ldots,(r_{\ell},s_{\ell})$. 
\end{conjecture}

In a recent paper \cite{KivTsai} Kivinen and Tsai proved the following:
\begin{theorem}[\cite{KivTsai}]
\label{thm: KivTsai}
The specialization of $(f_1, e_{r_1\cdots r_\ell})$ at $T=1$ from Conjecture \ref{conj: CD} is a polynomial in $Q$ which agrees with the point count of the compactified Jacobian over a finite field $\mathbb{F}_Q$.
\end{theorem}

\begin{remark}
A priori, $(f_1, e_{r_1\cdots r_\ell})$ is a rational function in $Q$ and $T$. In course  of the proof of Theorem \ref{thm: KivTsai}, Kivinen and Tsai proved using Compositional Rational Shuffle Theorem  that its specialization at $T=1$ is indeed well defined. 
\end{remark}

In the case of generic curves \eqref{eq: def generic} we have
$$
y=x^{\frac{m}{n}}+\lambda x^{\frac{md+1}{nd}}+\ldots,\ m_1=m,\ r_1=n,\ m_2=md+1,\ r_2=d.
$$ 
This means that $s_1=m$ and $s_2=1$. To follow the above procedure, we first need to compute the operator
$\gamma_{r_2,s_2}(p_1)$ and the corresponding symmetric function $f_2$. By \cite[Corollary 6.5]{GN} we have
$$
f_2=\gamma_{d,1}(p_1)(1)= P_{d,1}(1)=e_d.
$$
Therefore the next symmetric function is 
$$
f_1=\gamma_{n,m}(e_d)(1),
$$
and $(f_1,e_{nd})=C_{nd,md}(Q,T)$ by Theorem \ref{thm: daha}.

We conclude that for generic curves Conjecture \ref{conj: CD} is true and follows from Theorems \ref{thm: intro main} and  \ref{thm: daha}.

\begin{remark}
Strictly speaking, we need to consider the special cases $d=1$ and $n=1$ where the singularity has only one Puiseux pair separately. For $d=1$ we have $e_1=p_1$, so $f_1=\gamma_{n,m}(p_1)(1)$ as expected.

For $n=1$ we have one Puiseux pair $(d,md+1)$. Since $\nabla (1)=1,$ we have
$$
f_1=\gamma_{1,m}(e_d)(1)=\nabla^m (e_d \cdot \nabla^{-m}(1))=\nabla^m e_d = \nabla^{m}(P_{d,1}(1))=P_{d,md+1}(1).
$$ 
\end{remark}

\section{Curves with two Puiseux pairs}

\subsection{Cabled Dyck paths and $s$-admissible subsets}

In this section we consider arbitrary plane curve singularities with two Puiseux pairs:
$$
(x(t),y(t))=(t^{nd},t^{md}+\lambda t^{md+s}+\ldots),\quad \lambda\neq 0,
$$
where $\GCD(m,n)=1$ and $\GCD(d,s)=1$. As above, any $\cO_C$-module corresponds to a $(dn,dm)$-invariant subset $\Delta\subset \Z_{\ge 0}$.  Let $J_{\Delta}$ denote the corresponding stratum in the compactified Jacobian.

\begin{definition}
\label{def: s admissible}
An $(nd,md)$-invariant subset $\Delta$ is called $s$-admissible if $s$ is not suspicious for $\Delta$.
\end{definition}

 For $s=1$ this agrees with Definition \ref{def: admissible}. 
We conjecture the following weaker version of Theorem \ref{thm: intro main}.

\begin{conjecture}
\label{conj: chi}
The Euler characteristic of the stratum $J_{\Delta}$ is given by
$$
\chi(J_{\Delta})=\begin{cases}
1 & \text{if}\ \Delta\ \text{is}\ s\text{-admissible},\\
0 & \text{if}\ \Delta\ \text{is not}\  s\text{-admissible}.
\end{cases}
$$
\end{conjecture}

For $s=1$ the conjecture clearly follows from Theorem \ref{thm: intro main}. Piontkowski \cite{Piont} constructed cell decompositions for several cases with $s>1$ (see Theorem \ref{thm: piont} below for details), and the conjecture follows from his results in these cases.

\begin{remark}
Unlike the case $s=1$, the geometry of the strata $J_{\Delta}$ might be more complicated. In particular, \cite[Appendix A]{CheP} lists some non-affine strata for the singularity $(t^6,t^9+t^{13})$ which corresponds to $(m,n)=(2,3),(d,s)=(3,4)$. Some of the strata $J_{\Delta}$ are isomorphic to a union of two affine spaces glued along a codimension 1 subspace - these have the Euler characteristic 1 and correspond to $4$-admissible $\Delta$. Other strata are isomorphic to $\C^*\times \C^N$ -  these have the Euler characteristic 0 and correspond to $\Delta$ which are not $4$-admissible.  
\end{remark}
 
In particular, Conjecture \ref{conj: chi} would imply that the Euler characteristic of the compactified Jacobian equals the number of $s$-admissible subsets.To understand this number, we need the following combinatorial definition.

\begin{definition}
\label{def: cabled Dyck}
A cabled Dyck path with parameters $(n,m),(d,s)$ is the following collection of data:
\begin{itemize}
\item An $(d,s)$ Dyck path $P$ called the {\em pattern}. 
\item The numbers $v_1(P),\ldots,v_k(P)$ parametrizing the lengths of vertical runs in $P$, so that $$v_1(P)+\ldots+v_k(P)=d.$$
\item An arbitrary tuple of $(v_i(P)n,v_i(P)m)$-Dyck paths, $1\le i\le k$.
\end{itemize}
\end{definition}
 
Clearly, the total number of cabled Dyck paths equals
$$
c_{(m,n),(d,s)}=\sum_{P\in \Dyck(d,s)}c_{v_1(P)n,v_1(P)m}\cdots c_{v_k(P)n,v_k(P)m}
$$
where $c_{a,b}=C_{a,b}(1,1)$ is the number of Dyck paths in $a\times b$ rectangle.

\begin{example}
For $s=1$ there is only one $(d,s)$-Dyck paths $P$ with $v_1(P)=d$, so $$c_{(n,m),(d,1)}=c_{dn,dm}.$$ 

For $d=2$ there is one pattern  with $v_1(P)=2$, and $\frac{s-1}{2}$ patterns with $v_1(P)=v_2(P)=1$. Therefore
$$
c_{(n,m),(2,s)}=c_{2n,2m}+\frac{s-1}{2}c_{n,m}^2.
$$
\end{example}

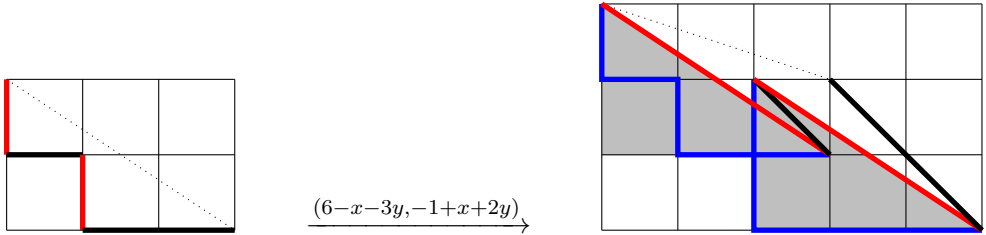
\begin{figure}[ht!]
\begin{tikzpicture}
\draw (0,0)--(0,2);
\draw (1,0)--(1,2);
\draw (2,0)--(2,2);
\draw (3,0)--(3,2);
\draw (0,0)--(3,0);
\draw (0,1)--(3,1);
\draw (0,2)--(3,2);
\draw [dotted] (0,2)--(3,0);

\draw [line width=2,red] (0,2)--(0,1);
\draw [line width=2] (0,1)--(1,1);
\draw [line width=2,red] (1,1)--(1,0);
\draw [line width=2] (1,0)--(3,0);
\end{tikzpicture}
\qquad
$\xrightarrow{(6-x-3y,-1+x+2y)}$
\qquad
\begin{tikzpicture}
\filldraw[lightgray] (0,1)--(0,3)--(3,1)--(0,1);
\filldraw[lightgray] (2,0)--(2,2)--(5,0)--(2,0);

\draw (0,0)--(0,3);
\draw (1,0)--(1,3);
\draw (2,0)--(2,3);
\draw (3,0)--(3,3);
\draw (4,0)--(4,3);
\draw (5,0)--(5,3);
\draw (0,0)--(5,0);
\draw (0,1)--(5,1);
\draw (0,2)--(5,2);
\draw (0,3)--(5,3);
\draw [dotted] (0,3)--(3,2);

\draw [line width=2,blue] (0,3)--(0,2)--(1,2)--(1,1)--(3,1);
\draw [line width=2,blue]  (2,2)--(2,0)--(5,0);

\draw [line width=2,red] (0,3)--(3,1);
\draw [line width=2] (3,1)--(2,2);
\draw [line width=2,red] (2,2)--(5,0);
\draw [line width=2] (5,0)--(3,2);
\end{tikzpicture}
\caption{Cabled Dyck path for the singularity $(t^4,t^6+t^9)$: here $(d,s)=(n,m)=(2,3)$. The pattern $P$ (left) in $(d,s)$ rectangle has $v_1=v_2=1$ and $h_1=1,h_2=2$. On the right, the same pattern after affine transformation (red and black) and a pair of $(v_im,v_in)$ Dyck paths (blue).}
\label{fig: cabled Dyck}
\end{figure}

\begin{remark}
\label{rem: sawtooth}
One can interpret the data of a cabled Dyck path geometrically as follows. First, draw the pattern $P$ in the $(d,s)$ rectangle. Then, apply an affine transformation in $\mathrm{SL}(2,\Z)$ which sends the vector $(0,1)$ to $(-m,n)$, it will transform a Dyck path to a sawtooth-like
 shape where the vertical steps $v_i$ in $P$ become diagonals in some $(v_im,v_in)$ rectangles, shown in red in Figure \ref{fig: cabled Dyck}. Now 
$(v_im,v_in)$-Dyck paths can be inscribed in $(v_im,v_in)$ rectangles below these diagonals (that is, contained in gray areas in Figure \ref{fig: cabled Dyck}). For example, one can consider the blue paths in Figure \ref{fig: cabled Dyck}.

Note that the gray areas are overlapping, so the $(v_im,v_in)$-Dyck paths might intersect in a complicated way. 
\end{remark}

\begin{theorem}
\label{thm: s admissible}
There is a bijection between the sets of $s$-admissible $(dn,dm)$-invariant  subsets and cabled Dyck paths with parameters $(n,m),(d,s)$.
\end{theorem}

\begin{proof}
Let $\Delta=\cup_{j=0}^{d} \Delta_j$ be  $(dn,dm)$-invariant subset. 

We consider another subset 
$$
\overline{\Delta}=\cup_{j=0}^{d} \overline{\Delta}_j,\quad \overline{\Delta}_j=\Delta_{sj}-(s-1)j.
$$
Note that $\overline{\Delta}$ might have negative elements.
Since $\Delta_{sj}$ is supported in remainder $sj\mod d$, $\overline{\Delta}_j$ is supported in remainder $j\mod d$.

We claim that $\Delta$ is $s$-admissible if and only if $1$ is not $j$-suspicious for $\overline{\Delta}$, $0\le j\le d-2$.
Indeed, for $0\le j\le d-2$ it is clear from Lemma \ref{lem:susp} that $s$ is $(sj)$-suspicious for $\Delta$ if and only if
$$
\Delta_{sj+s}\subset \Delta_{sj}+md+nd+s\ \Leftrightarrow\ \overline{\Delta}_{j+1}\subset \overline{\Delta}_{j}+1,
$$
that is, $1$ is $j$-suspicious for $\overline{\Delta}$. Since $\Delta$ is $0$-normalized, we have $\min \Delta_0=0$ and $\min \Delta_{s(d-1)}>0$ so by Lemma \ref{lem: mins} $s$ is not $s(d-1)$-suspicious.

Next, we define the cabled Dyck path corresponding to $\overline{\Delta}$. Note that $\overline{\Delta}_0=\Delta_0$, so $\min \overline{\Delta}_0=0$.
Suppose that for $j=0,\ldots,t_1-1$ we have $\min \overline{\Delta}_j\ge 0$ but $\min \overline{\Delta}_{t_1}<0$, then we define $v_1(P)=t_1$. By Theorem \ref{thm: unique admissible}  the collection of subsets $\overline{\Delta}_0,\ldots,\overline{\Delta}_{t_1-1}$ defines a unique $1$-admissible collection of $(m,n)$-Dyck paths in the corresponding equivalence class, and hence by Theorem \ref{thm: GMVDyck} a unique $(t_1m,t_1n)$-Dyck path. Note that by Lemma \ref{lem: mins} the number $1$ is not $t_1$-suspicious for $\overline{\Delta}$.

Next, suppose that for $j=t_1,\ldots,t_2-1$ we have $\min \overline{\Delta}_j\ge \min \overline{\Delta}_{t_1}$ but $\min \overline{\Delta}_{t_2}<\min \overline{\Delta}_{t_1}$, then we define $v_2(P)=t_2-t_1$. As above, $\overline{\Delta}_{t_1},\ldots,\overline{\Delta}_{t_2-1}$ define a unique $((t_2-t_1)m,(t_2-t_1)n)$-Dyck path and so on. 

This defines a collection of vertical runs $v_i(P)$ with $\sum v_i(P)=d$ and a collection of Dyck paths in the $(v_i(P)m\times v_i(P)n)$ rectangles. To complete the bijection, we need to define the horizontal steps of the pattern $P$. Consider a sequence of integers $q_i$ defined by the equation
\begin{equation}
\label{eq: shift delta bar}
\min \overline{\Delta}_{t_i}=t_i-q_{i}d,
\end{equation}
then by construction we have 
$$
t_{i}-q_{i}d<t_{i-1}-q_{i-1}d,\quad (q_{i}-q_{i-1})d>t_{i}-t_{i-1}>0,
$$
hence $q_{i}>q_{i-1}$ and we can define $h_i(P)=q_{i}-q_{i-1}>0$. On the other hand,
$$
t_i-q_id=\min \Delta_{st_i}-st_i+t_i\ge t_i-st_i,
$$
therefore 
$$
(h_1(P)+\ldots+h_i(P))d=q_id\le st_i=s(v_1(P)+\ldots+v_i(P))
$$
and a lattice path $P$ with vertical steps $v_i$ and horizontal steps $h_i$ stays below the diagonal. 

Conversely, assume that we are given a cabled Dyck path with pattern $P$ and a collection of $(v_i(P)m,v_i(P)n)$ Dyck paths.
Using the vertical and horizontal steps of $P$, we can reconstruct $t_i$ and $q_i$ as above and set up $t_0=q_0=0$. By Theorems \ref{thm: GMVDyck} and \ref{thm: unique admissible} we can translate each $(v_i(P)m,v_i(P)n)$ Dyck path into a unique $1$-admissible collection of $(dn,dm)$ invariant subsets $\overline{\Delta}_{t_{i-1}},\ldots,\overline{\Delta}_{t_{i}-1}$ normalized by \eqref{eq: shift delta bar}. This determines $\overline{\Delta}$ completely. 

Finally, since $\GCD(d,s)=1$, the values of $sj\mod d$ run over all remainders modulo $d$ once, and we can uniquely define $\Delta_j$ by
$$
\Delta_{sj}=\overline{\Delta}_j+sj-j.
$$
By reversing the argument above, $\Delta$ is $s$-admissible. Assume $t_i\le j<t_{i+1}$, then
$$
\min \Delta_{sj}=\min \overline{\Delta}_{j}+(s-1)j\ge \min \overline{\Delta}_{t_i}+(s-1)t_i=t_i-q_{i}d+(s-1)t_i=st_i-q_id>0,
$$
and $\Delta$ is $0$-normalized.
\end{proof}

\begin{example}
Let us reconstruct the $(4,6)$-invariant 3-admissible subset $\Delta$  corresponding to the data of cabled Dyck paths in Figure \ref{fig: cabled Dyck} (see also Remark \ref{rem: sawtooth}). We have $(d,s)=(n,m)=(2,3)$ and $v_1=v_2=1$ and $h_1=1,h_2=2$. This means that $t_1=1,q_1=1$ and
$$
\min \overline{\Delta}_0=0,\ \min \overline{\Delta}_1=-1.
$$
Furthermore, it is easy to see that $\overline{\Delta}_0$ (corresponding to the top blue  path in Figure \ref{fig: cabled Dyck}) up to a shift 
coincides with $2\Z_{\ge 0}=\{0,2,4,6,\ldots\}$ while  $\overline{\Delta}_1$ (corresponding to the bottom blue  path in Figure \ref{fig: cabled Dyck}) up to a shift  coincides with $2(\Z_{\ge 0}\setminus \{2\}=\{0,4,6,\ldots\}$. Therefore
$$
\overline{\Delta}_0=\{0,2,4,6,\ldots\},\ \overline{\Delta}_1=\{-1,3,5,\ldots\}.
$$
Now $\Delta_0=\overline{\Delta}_0=\{0,2,4,6,\ldots\},\ \Delta_1=\overline{\Delta}_1+3-1=\{1,5,7,\ldots\}$
and
$$
\Delta=\{0,1,2,4,5,6,\ldots\}=\Z_{\ge 0}\setminus\{3\}.
$$
\end{example}

\begin{lemma}
The specialization of Conjecture \ref{conj: CD} at $Q=T=1$ for singularities of the form 
$$
(x(t),y(t))=(t^{nd},t^{md}+\lambda t^{md+s}+\ldots),\ \lambda\neq 0
$$
agrees with the number of cabled Dyck paths with parameters $(n,m),(d,s)$.
\end{lemma}

\begin{proof}
We follow the notations of Section \ref{sec: CD}, there are two characteristic pairs $(n,m)$ and $(d,s)$.
We start from a symmetric function $f_2=P_{d,s}(1)$ which specializes to
$$
f_2|_{Q=T=1}=\sum_{P}e_{v_1(P)}\cdots e_{v_k(P)}.
$$
At $Q=T=1$ the operators $\gamma_{n,m}(f)$ are simply multiplication operators on $\Lambda$ (see e.g. \cite{KivTsai,NegutInt}), so 
$$
\gamma_{n,m}(fg)_{Q=T=1}(1)=\gamma_{m,n}(f)_{Q=T=1}(1)\cdot \gamma_{m,n}(g)_{Q=T=1}(1) .
$$
 Let $F_{nd,md}=\gamma_{n,m}(e_d)(1)|_{Q=T=1}$, then
$$
f_1|_{Q=T=1}=\gamma_{n,m}(f_2)(1)|_{Q=T=1}=\sum_{P}\left[\gamma_{n,m}(e_{v_1(P)})\cdots \gamma_{n,m}(e_{v_k(P)})(1)\right]|_{Q=T=1}=
$$
$$
\sum_{P}F_{v_1(P)n,v_1(P)m}\cdots F_{v_k(P)n,v_k(P)m}.
$$
Finally, by Theorem \ref{thm: daha} we have $(F_{nd,md},e_{nd})=C_{nd,md}(1,1)=c_{nd,md}$ and the result follows. 
See \cite{KivTsai} for related computations (at $Q=1$) and more details.
\end{proof}

\subsection{Piontkowski's results}
\label{sec: piont}

For the reader's convenience, we summarize the main results of \cite{Piont} for singularities with two Puiseux pairs using the terminology of this paper.  

\begin{theorem}[\cite{Piont}]
\label{thm: piont}
Suppose $d=2$, $(n,m)=(2,q),(3,4)$ or $(3,5)$ and $s$ is an arbitrary odd number. Then:
\begin{itemize}
\item[a)] If $\Delta$ is not $s$-admissible then $J_{\Delta}$ is empty.

\item[b)] If $\Delta$ is $s$-admissible then $J_{\Delta}$ is isomorphic to an affine space (and the formula for its dimension is given in \cite{Piont}).

\item[c)] The Euler characteristic of $\overline{JC}$ is equal to the number of $s$-admissible subsets, which is given by the following table:
\begin{center}
\begin{tabular}{|c|c|}
\hline
& \\
$(2,q)$ & $\frac{(q+1)(q^2+5q+3)}{12}+\frac{(q+1)^2}{8}(2q+s)$\\ 
 & \\
\hline
& \\
$(3,4)$ & $\frac{229}{2}+\frac{25}{2}(8+s)$ \\ 
& \\
\hline
& \\
$(3,5)$ &  $\frac{511}{2}+\frac{49}{2}(10+s)$ \\
& \\
\hline
\end{tabular}
\end{center}

\end{itemize}
\end{theorem}

Observe that Theorem \ref{thm: piont}(c) can be compactly written as follows:

\begin{proposition}
Consider a plane curve singularity $C$ with Puiseux exponents $(2n,2m,2m+s)$ where $(n,m)=(2,q),(3,4)$ or $(3,5)$. Then 
$$
\chi(\overline{JC})=c_{2n,2m}+(c_{n,m})^2\cdot \frac{s-1}{2}=c_{(n,m),(2,s)}.
$$
\end{proposition}

For $s=1$ we recover $\chi(\overline{JC})=c_{2n,2m}$, in agreement with Theorem \ref{thm: intro main}.

\begin{proof}
We have $c_{2,q}=\frac{q+1}{2}$, $c_{3,4}=5$ and $c_{3,5}=7$. Using e.g. the formula in Corollary \ref{cor: bizley}, we compute
$$
c_{4,2q}=\frac{(2q+3)!}{4!(2q)!}+\frac{1}{2}\left(\frac{q+1}{2}\right)^2=\frac{(q+1)(2q+1)(2q+3)}{12}+\frac{(q+1)^2}{8},
$$
$$
c_{6,8}=\frac{1}{14}\binom{14}{6}+\frac{5^2}{2}=\frac{429}{2}+\frac{25}{2}=227,
$$
$$
c_{6,10}=\frac{1}{16}\binom{16}{6}+\frac{7^2}{2}=\frac{1001}{2}+\frac{49}{2}=525.
$$
Now the result is easy to check using the above table.
\end{proof}

 \begin{example}
\label{ex: 467}
The simplest singularity with two Puiseux pairs is $(t^4,t^6+t^7)$. The Dyck paths in $4\times 6$ rectangle can be thought of subdiagrams of $\lambda_{4,6}=(4,3,1)$ which can be listed as following:

\begin{center}
\begin{tabular}{c|c}

   $|D|$  & $D$ \\
   \hline
    $0$     &  $(0,0,0)$\\
\hline
 $1$     &  $(1,0,0)$\\
\hline
 $2$    &  $(2,0,0),(1,1,0)$\\
\hline
 $3$    &  $(3,0,0),(2,1,0),(1,1,1)$\\
\hline
 $4$     &  $(3,1,0),(2,2,0),(2,1,1),(4,0,0)$\\
\hline
 $5$     &  $(3,2,0),(3,1,1),(2,2,1),(4,1,0)$\\
\hline
 $6$     &  $(3,2,1),(3,3,0),(4,2,0),(4,1,1)$\\
\hline
 $7$     &  $(3,3,1),(4,3,0),(4,2,1)$\\
\hline
 $8$     &  $(4,3,1)$\\
\hline
\end{tabular}
\end{center}
The generating function equals
$$
P_{4,6}(T)=\sum_{D}T^{|D|}=1+T+2T^2+3T^3+4T^4+4T^5+4T^6+3T^7+T^8
$$
which agrees with the Poincar\'e polynomial of $\overline{JC}$ computed in \cite{Piont}, as in Corollary \ref{cor: Poincare area intro}. The Euler characteristic equals
$$
23=\frac{1}{10}\binom{10}{4}+\frac{1}{2}\left(\frac{1}{5}\binom{5}{2}\right)^2=21+\frac{2^2}{2}.
$$
\end{example}

\end{document}